\documentclass[11pt,a4paper]{article} 

\usepackage{amsmath,amsthm,amssymb}
\usepackage[top=3cm,bottom=3cm,left=3cm,right=3cm]{geometry} 
\usepackage[all]{xy} 
\usepackage[T1]{fontenc} 
\usepackage{textcomp} 
\usepackage{slashed}
\usepackage{hyperref}

 \usepackage{beton} \usepackage{euler} \usepackage[OT1]{fontenc}
\DeclareMathOperator{\End}{End}

\DeclareMathOperator{\tr}{Tr}

\DeclareMathOperator{\ind}{ind}

\DeclareMathOperator{\Td}{Td}

\DeclareMathOperator{\Id}{Id}

\DeclareMathOperator{\Spin}{Spin}

\DeclareMathOperator{\U}{U}
\DeclareMathOperator{\SU}{SU}

\DeclareMathOperator{\Map}{Map} 

\DeclareMathOperator{\SW}{SW}

\newcommand{\R}{\mathbb R}
\newcommand{\C}{\mathbb C}

\newcommand{\N}{\mathbb N}

\newcommand{\G}{\mathcal G}

\newcommand{\A}{\mathcal A}

\newcommand{\E}{\mathcal E}

\newcommand{\Spinc}{\Spin^c}
\newcommand{\diff}{\mathrm{d}}

\newcommand{\dvol}{\mathrm{dvol}}

\newcommand{\su}{\mathfrak{su}}

\newcommand{\odd}{\mathrm{odd}}
\newcommand{\even}{\mathrm{even}}

\renewcommand{\mod}{\text{\rm mod}\,}
\renewcommand{\P}{\mathbb P}

\renewcommand{\u}{\mathfrak{u}}
\renewcommand{\tilde}{\widetilde}

\theoremstyle{plain}
	\newtheorem{theorem}{Theorem}
	\newtheorem{proposition}[theorem]{Proposition}
	\newtheorem{lemma}[theorem]{Lemma}
	\newtheorem{corollary}[theorem]{Corollary}

\theoremstyle{definition}
	\newtheorem{definition}[theorem]{Definition}
	\newtheorem{remark}[theorem]{Remark}
	\newtheorem{remarks}[theorem]{Remarks}

\theoremstyle{plain}
	\newtheorem*{theorem*}{Theorem}
	\newtheorem*{proposition*}{Proposition}
	\newtheorem*{lemma*}{Lemma}
	\newtheorem*{corollary*}{Corollary}
	\newtheorem*{conjecture*}{Conjecture}
\theoremstyle{definition}
	\newtheorem*{definition*}{Definition}
	\newtheorem*{remark*}{Remark}
	\newtheorem*{remarks*}{Remarks}

\numberwithin{equation}{section}
\numberwithin{theorem}{section}

\usepackage[nottoc,numbib]{tocbibind}

\begin{document}

\title{Seiberg--Witten equations in all dimensions}
\author{Joel Fine and Partha Ghosh}

\maketitle

\vfill

\begin{abstract}
Starting with an $n$-dimensional oriented Riemannian manifold with a $\Spinc$-struc\-ture, we describe an elliptic system of equations which recover the Seiberg--Witten equations when $n=3,4$. The equations are for a $\U(1)$-connection $A$ and spinor $\phi$, as usual, and also an odd degree form $\beta$ (generally of inhomogeneous degree). From $A$ and $\beta$ we define a Dirac operator $D_{A,\beta}$ using the action of $\beta$ and $*\beta$ on spinors (with carefully chosen coefficients)  to modify $D_A$. The first equation in our system is $D_{A,\beta}(\phi)=0$. The left-hand side of the second equation is the principal part of the Weitzenb\"ock remainder for $D^*_{A,\beta}D_{A,\beta}$. The equation sets this equal to $q(\phi)$, the trace-free part of projection against $\phi$, as is familiar from the cases $n=3,4$. In dimensions $n=4m$ and $n=2m+1$, this gives an elliptic system modulo gauge. To obtain a system which is elliptic modulo gauge in dimensions $n=4m+2$, we use two spinors and two connections, and so have two Dirac and two curvature equations, that are then coupled via the form $\beta$. We also prove a collection of a~priori estimates for solutions to these equations. Unfortunately they are not sufficient to prove compactness modulo gauge, instead leaving the possibility that bubbling may occur.
\end{abstract}

\vfill

\tableofcontents

\newpage

\section{Introduction}
Let $(M,g)$ be a Riemannian manifold of dimension $n$, which admits a $\Spinc$-structure. The aim of this article is to introduce an elliptic system of equations on $M$ which generalise the Seiberg--Witten equations in the cases $n=3,4$. (These equations were first written down, in 4-dimensions, by Seiberg and Witten \cite{Seiberg-Witten}.)

We begin by fixing notation. Write $S \to M^n$ for the spin bundle of the $\Spinc$-structure. When $n$ is even, $S =S_+ \oplus S_-$ splits into subbundles of positive and negative spinors. We write $c \colon \Lambda^* \to \End(S)$ for the Clifford action of differential forms on spinors. We follow the conventions of \cite{Friedrich}. In particular, (real) 1-forms act as skew-Hermitian endomorphisms. Meanwhile, in dimension $n=2m$, the volume form satisfies $i^mc(\dvol)) = \pm 1$ on $S_{\pm}$ whilst in dimension $n=2m-1$, $i^m c(\dvol) = 1$ on all of $S$. 

Let $L$ denote the line bundle associated to the $\Spinc$ structure. Recall that if two $\Spinc$ structures with spinor bundles $S,S'$ are related by $S' = S \otimes E$, for some line bundle $E$, then the associated line bundles satisfy $L' = L \otimes E^2$. Let $\A$ denote the set of unitary connections in $L$. A connection $A \in \A$ singles out a unitary connection $\nabla_A$ in $S$ and hence a Dirac operator $D_A$ on sections of $S$. . 

When $n$ is even, a spinor $\phi \in S_\pm$ defines a trace-free Hermitian endomorphism $E_\phi \colon S_\pm \to S_\pm$ via
\begin{equation}
E_\phi(\psi)= \left\langle \psi,\phi \right\rangle\phi - \frac{1}{r} |\phi|^2 \psi.
\label{Ephi}
\end{equation}
where $r$ is the rank of $S_\pm$. When $n$ is odd and $\phi \in S$ we use $E_\phi$ to denote the analogous trace-free endomorphism of $S$, where now $r$ is the rank of $S$. 

Before giving the $n$-dimensional version of the Seiberg--Witten equations, we first recall the three and four dimensional cases, highlighting the features we will generalise. In dimension~4, Clifford multiplication gives an isomorphism 
\begin{equation}\label{S20-in-dimension-4}
c \colon i \Lambda^2_+ \to i\su(S_+)
\end{equation}
between the imaginary self-dual 2-forms and trace-free Hermitian endomor\-phisms of $S_+$. Given $\phi \in \Gamma(S_+)$, we write $q(\phi) \in i\Omega^2_+$ for the imaginary self-dual 2-form corresponding to $E_\phi$ under~\eqref{S20-in-dimension-4}. The Seiberg--Witten equations for $A \in \A$ and $\phi \in \Gamma(S_+)$ are:
\begin{align}
D_A \phi &=0,\label{SW4-Dirac}\\
F_A^+ & = q(\phi)\label{SW4-curvature}.
\end{align}
The gauge group $\G = \Map(M, S^1)$ acts on $(A,\phi)$ by pull-back and this action preserves the space of solutions to the equations. 

The whole 4-dimensional story is now based on two crucial facts. Firstly, these equations are \emph{elliptic modulo gauge}. This is ultimately because the linearisation of the map $A \mapsto F_A^+$ combined with Coulomb gauge, to work modulo the gauge action, gives the truncated de~Rham complex, in the form
\[
 \Omega^1 \stackrel{\diff^+ + \diff^*}{\longrightarrow} \Omega^0 \oplus \Omega^2_+.
\]
Secondly, the space of solutions to the equations are \emph{compact modulo gauge}. This follows ultimately from the fact that the left-hand side of the curvature equation~\eqref{SW4-curvature} is directly related to the Weitzenb\"ock formula\footnote{The formula relating the Dirac Laplacian to the rough Laplacian is due to Schr\"odinger \cite{Schrodinger} in 1932, and was subsequently rediscovered by Lichnerowicz \cite{Lichnerowicz} in 1962.  The analogous formula relating a generalised Laplacian to a rough Laplacian is seemingly due to Weitzenb\"ock. Historically a better name for~\eqref{Weitzenbock-4D} would probably be the ``Schr\"odinger--Lichnerowicz formula'' but we follow a relatively common practice by calling this and similar equations ``Weitzenb\"ock formulae''.}:
\begin{equation}
D^2_A
-
\nabla_A^*\nabla_A 
	=
		\frac{s}{4}
		+
		\frac{1}{2} c(F_A^+),
\label{Weitzenbock-4D}
\end{equation}
where $s$ is the scalar curvature of $(M,g)$. 

To summarise: \emph{on a 4-manifold, using Dirac operators $D_A$ parametrised by $A \in \A$ ensures that prescribing the Weitzenb\"ock remainder gives an elliptic system modulo gauge.}

In dimension~3, meanwhile, Clifford multiplication gives an isomorphism 
\begin{equation}
c \colon i \Lambda^2 \to i\su(S)
\label{Clifford-isomorphism-3D}
\end{equation}
Given $\phi \in \Gamma(S)$, we now write $q(\phi) \in i \Omega^2$ for the imaginary 2-form corresponding to $E_\phi$ under~\eqref{Clifford-isomorphism-3D}. The 3-dimensional Seiberg--Witten equations for $A \in \A$, $\phi \in \Gamma(S)$ and $\beta \in \Omega^3$ are:
\begin{align}
\left(D_A + i c( *\beta) \right) \phi
	&= 0,
		\label{SW3-Dirac}\\
F_A - 2i \diff^* \beta
	&= q(\phi).
		\label{SW3-curvature}
\end{align}
One often sees these equations written with $\beta=0$. This is because when $\phi$ is not identically zero and $M$ is compact, one can check that the equations actually force $\beta=0$. It is also perhaps more common to see the equations with the 3-form $\beta$ replaced by the function $* \beta$. We choose to use $\beta$ as the variable because it fits more cleanly with our generalisation. 

The equations are elliptic modulo gauge because the linearisation of $(A,\beta) \mapsto F_A + i \diff^* \beta$, in combination with Coulomb gauge, produces the de~Rham complex, in the form 
\[
\Omega^1 \oplus \Omega^3 \stackrel{\diff + \diff^*}{\longrightarrow} \Omega^0 \oplus \Omega^2.
\]
Just as in the 4-dimensional case, the curvature equation~\eqref{SW3-curvature} is related to a Weitzenb\"ock remainder. This time
\begin{equation}
D_{A,\beta}^*D_{A,\beta} - \nabla^*_{A}\nabla_{A}
	=
		\frac{s}{4}
		+
		\frac{1}{2}c(F_A)
		-
		i \diff^*\beta
		+
		|\beta|^2.
\label{Weitzenbock-3D}
\end{equation}
Here $D_{A,\beta} = D_A + ic(*\beta)$ is the Dirac operator appearing in~\eqref{SW3-Dirac}. 

To summarise: \emph{on a 3-manifold, using Dirac operators $D_{A,\beta}$ parametrised by $A \in \A$ and $\beta \in\Omega^3$ ensures that prescribing the principal part of the Weitzenb\"ock remainder gives an elliptic system modulo gauge.}

These features are what we generalise to give Seiberg-Witten equations in arbitrary dimensions: we consider a family of Dirac operators $D_{A,\beta}$ parametrised by $A \in \A$ and an a certain choice of odd degree form $\beta$ (of inhomogeneous degree). The odd degree forms ensure that prescribing the principal part of the Weitzenb\"ock remainder of $D_{A,\beta}$ is an elliptic system, modulo gauge. When combined with the Dirac equation, the equations have the potential for good analytic properties. We will give some analytic results in this direction. We stress from the outset however that when $n >4$ our results are not sufficient to prove that the solution space is compact. Instead they leave open the possibility that ``bubbling'' can occur.

An obvious question is what purpose might these higher-dimensional Seiberg--Witten equations serve? In higher dimensions, there is no need for a gauge theoretic approach to study smooth structures since, for example, the $h$-cobordism theorem holds \cite{Smale}. Instead, one might speculate that higher dimensional Seiberg--Witten equations could prove useful when studying manifolds with geometric structures. This tentative line of thought is inspired by Taubes' work on symplectic 4-manifolds. Taubes proved that for a compact symplectic 4-manifold with $b_+>1$ the Seiberg--Witten invariant for the canonical $\Spinc$ structure is always 1 \cite{Taubes1}. This gives an obstruction to the existence of symplectic structures. There is no known obstruction in higher dimensions, beyond the most obvious that there must be a degree~2 cohomology class with non-zero top power. Even deeper is Taubes' Theorem that the Seiberg--Witten invariants are equal to the Gromov--Witten invariants \cite{Taubes2,Taubes3,Taubes4}. In particular, for a symplectic manifold with $b_+>1$, the canonical class is always represented by a $J$-holomorphic curve. In higher dimensions there are no known general existence results of this kind for $J$-holomorphic curves. It is, of course, very speculative to hope that these higher dimensional Seiberg--Witten equations could tell us something about higher dimensional symplectic manifolds (especially in light of the fact that the analysis appears much more complicated; see \S\ref{analysis}!), but at least it does not seem completely impossible. 

Dirac operators of the form $D +c(\beta)$ where $\beta \in \Omega^3$ has degree \emph{three} have appeared in several contexts, going back at least as far as Bismut's pioneering work on Dirac operators associated to metric connections with torsion \cite{Bismut}. This same paper also gives a Weitzenb\"ock formula which is a particular case of the various formulae proved here. Since Bismut's work there has been a huge amount of work on these particular Dirac operators; so much so that it is futile to give a survey here. To the best of our knowledge, however, there is only one paper which considers this kind of Dirac opertaor in the context of the Seiberg--Witten equations, namely the work \cite{Tanaka} of Tanaka. Tanaka  formulates a version of the Seiberg--Witten equations on a \emph{symplectic} 6-manifold, which have some similarity to the equations described here. To write down Tanaka's equations one must first pick an almost complex structure compatible with the symplectic form. This is in contrast to our equations which need nothing more than a Riemannian metric and a $\Spinc$-structure. There is an interesting point in common however: Tanaka perturbs the Dirac operator by adding a $(0,3)$-form to it; this is similar to the point of view taken here, where in 6-dimensions we also perturb the Dirac operator, this time by an arbitrary 3-form. 

The article is structured as follows. In the remaining sections ~\S\S\ref{SW-odd}--\ref{SW-4m-2} of this introduction, we describe the $n$-dimensional Seiberg--Witten equations. In~\S\ref{5D-examples} we give examples of solutions to the equations in dimension~5. In~\S\ref{elliptic} we show that the equations are elliptic, modulo gauge, and compute their index. In~\S\ref{Weitzenbock} we show that the curvature equation is precisely the principal part of the Weitzenb\"ock remainder term. In~\S\ref{analysis} we exploit this to prove some preliminary a~priori estimates on solutions to the equations. Finally, in \S\ref{energy-identity} we use the Weitzenb\"ock formula to show that solutions to the $n$-dimensional Seiberg--Witten equations are absolute minima of a natural energy functional (just as happens in dimensions~3 and~4). 

In one sequel to this article, the second author will give several more examples of solutions to the equations \cite{Ghosh}.  In another sequel we will describe modified versions of these Seiberg--Witten equations in dimensions 6, 7 and 8, which make sense on manifolds with $\SU(3)$, $G_2$ and $\Spin(7)$ structures respectively \cite{Fine-Ghosh-Ragini}.

\subsection*{Acknowledgements}

It is a pleasure to acknowledge conversations with Patrick Weber and Ralph Klaasse on earlier versions of these equations. We are also grateful to Chengjian Yao whose careful reading of a draft helped eliminate errors. JF was supported by FNRS grant PDR T.0082.21 and an ``Excellence of Science'' grant 4000725. PG was supported by FNRS grants FRIA-B1-40004195 and FRIA-B2-40014665.

\subsection{The Seiberg--Witten equations in odd dimensions}\label{SW-odd}

In odd dimensions, $n = 2m+1$, it is relatively straightforward to generalise the 3-dimensional story. Only the notation becomes more complicated. To ease things a little, we define a function $s \colon \N \to \{1,i\}$ by
\[
s_k = \begin{cases}
	1 & \text{if } k \equiv 0 \text{ or } 3\, \mod 4\\
	i & \text{if } k \equiv 1 \text{ or } 2\, \mod 4	
	\end{cases}
\]
The point of this definition is that if $\beta_k \in \Omega^k$  then $s_k c(\beta_k)$ is a self-adjoint endomorphism of the spin bundle.

In dimension $2m+1$, Clifford multiplication gives an isomorphism
\begin{equation}
c \colon i \Lambda^2 \oplus \Lambda^4 \oplus \cdots \oplus s_{2m} \Lambda^{2m} \to i\su(S)
\label{Clifford-isomorphism-2m+1}
\end{equation}
Given $\phi \in S$, we write $q(\phi)$ for the differential form corresponding to $E_\phi$ under~\eqref{Clifford-isomorphism-2m+1}. We consider equations for $(A,\beta,\phi)$ where $A \in \A$, $\beta = \beta_3 + \beta_5 + \cdots + \beta_{2m+1}$ with $\beta_k \in \Omega^k$, and $\phi \in \Gamma(S)$. We set
\begin{equation}
D_{A,\beta}
	=
		D_A
		+
		\sum_{k=1}^{m-1} s_{2k+1} c(\beta_{2k+1}) 
		+ 
		\sum_{k=1}^mi s_{2m-2k}c(*\beta_{2k+1})
\label{Dirac-2m+1}
\end{equation}
(In fact, $c(*\beta_{2k+1})$ is equal to $c(\beta_{2k+1})$ multiplied by some power of $i$, depending on $k$ and $m$; the precise factor is in Lemma~\ref{Clifford-Hodge} below. We choose to write $D_{A,\beta}$ this way to make clear the self-adjoint and skew-adjoint parts.)
Write
\begin{align}
F_\beta
	&=
		2 \sum_{k=1}^{m-1}s_{2k+2} \diff \beta_{2k+1}\label{F-beta-2m+1},\\
C_\beta
	&=
		2\sum_{k=1}^{m}
	(-1)^{km + m+k + 1 + \frac{m(m+1)}{2}} s_{2k} \diff^* \beta_{2k+1}.
		\label{C-beta-2m+1}
\end{align}
The notation $F_\beta$ and $C_\beta$ is explained below in Remark~\ref{gerbes}. The key thing to keep in mind for now is that, up to some factors of $2i$, $F_\beta$ is essentially $\diff \beta$ whilst, again up to some factors of $2i$ and also some cumbursome signs (an artefact of how Clifford multiplication works), $C_\beta$ is essentially $\diff^*\beta$.  Also note that the factors of $s_{2k}$ ensure that both $c(F_\beta)$ and $c(C_\beta)$ are self-adjoint.

\begin{definition}\label{SW2m+1}
Let $M^{2m+1}$ be an oriented $2m+1$-dimensional manifold with $\Spinc$-structure. The $(2m+1)$-dimensional Seiberg--Witten equations for $(A,\beta,\phi)$ are:
\begin{align}
D_{A,\beta}(\phi)
	&=
		0,\label{SW2m+1-Dirac}\\
F_A + F_\beta + C_\beta
	&=
		q(\phi),\label{SW2m+1-curvature}
\end{align}
where $D_{A,\beta}$, $F_\beta$ and $C_\beta$ are given by~\eqref{Dirac-2m+1}, \eqref{F-beta-2m+1} and~\eqref{C-beta-2m+1} respectively.
\end{definition}

When $m=1$, $F_\beta = 0$, $C_\beta = - 2i \diff^*\beta_3$ and we recover the ordinary 3-dimensional Seiberg--Witten equations. 

The point of these equations is that, as is shown in \S\ref{Weitzenbock}, the Dirac operator $D_{A,\beta}$ has a Weitzenb\"ock formula of the form
\[
D^*_{A,\beta} D_{A,\beta} - \nabla^*_{A,\beta}\nabla_{A,\beta}
	=
		\frac{s}{4}
		+
		\frac{1}{2}c(F_A + F_\beta + C_\beta)
		+
		Q(\beta)
\]
where $\nabla_{A,\beta}$ is a unitary connection on $S$ determined by $A$ and $\beta$ and $Q(\beta)$ is a zeroth order term which is purely algebraic in $\beta$. (For example, when $m=1$ this is equation~\eqref{Weitzenbock-3D} above where $Q(\beta) = |\beta|^2$.) So~\eqref{SW2m+1-curvature} prescribes the principal part of the Weitzenb\"ock remainder. Moreover, the system is elliptic modulo gauge, essentially because the de Rham complex is elliptic. This is the reason behind the various factors for $c(\beta_{2k+1})$ and $c(*\beta_{2k+1})$ in~\eqref{Dirac-2m+1}: they are chosen precisely to make $\diff \beta$ and $\diff^*\beta$ appear in the Weitzenb\"ock remainder. (See \S\ref{elliptic} for the details.)

\begin{remark}\label{gerbes}
If we think of $\beta$ as a collection of connection $k$-forms, in the sense of $\U(1)$-gerbes (or ``$k$-form gauge fields'' as they are called in the physics literature) then, up to the various factors of~$i$ and~$2$, $F_\beta$ is the sum of the curvatures of the $\beta_k$. Meanwhile $C_\beta=0$ is the Coulomb gauge condition. With this in mind it is tempting to think of $D_{A,\beta}$ as a Dirac operator coupled to various connections on appropriate $\U(1)$-gerbes. 

One reason to want to do this is to put $F_\beta$ on a similar footing to $F_A$. In 4-dimensional Seiberg--Witten theory it is important to be able to vary the cohomology class $[F_A]$ by choosing different $\Spinc$ structures. In particular, for some choices there are no solutions. In the above description, however, $[F_\beta]=0$ is fixed. To get non-zero classes, one would need to interpret $\beta_{2k+1}$ as a connection in a $2k$-gerbe with non-zero characteristic class. However we have been unable to make sense of ``spinors with values in a gerbe'' or of the action of gerbe gauge-transformations in this setting (or ``$(k-1)$-form gauge transformations'' as they are sometimes called in the physics literature).
\end{remark}

\subsection{The Seiberg--Witten equations in dimension $4m$}\label{SW-4m}

We next give the direct generalisation of the 4-dimensional Seiberg--Witten equations to dimension $n=4m$. Here, Clifford multiplication gives the following isomorphism:
\begin{equation}
c \colon i \Lambda^2 \oplus \Lambda^4 \oplus \cdots \oplus s_{2m} \Lambda^{2m}_+ \to i\su(S_+)
\label{4m-Clifford-isomorphism}
\end{equation}
where $\Lambda^{2m}_+$ is the $+1$ eigenspace of $*$ acting on $\Lambda^{2m}$. Given $\phi \in S_+$, we write $q(\phi)$ for the form which corresponds under~\eqref{4m-Clifford-isomorphism} to $E_\phi \in i\su(S_+)$. We consider equations for $(A,\beta, \phi)$ where $A \in \A$, $\beta = \beta_3 + \beta_5 + \cdots + \beta_{2m-1}$ with $\beta_k \in \Omega^k$ and $\phi \in \Gamma(S_+)$. We set
\begin{equation}
D_{A,\beta} = D_A +  \sum_{k=1}^{m-1} \Big( s_{2k+1} c (\beta_{2k+1}) + s_{4m-2k-1} c(*\beta_{2k+1} )\Big).
\label{Dirac-4mD}
\end{equation}
This is a self-adjoint operator on spinors, which swaps chirality, $D_{A,\beta} \colon \Gamma(S_+) \to \Gamma(S_-)$. Write
\begin{align}
F^+_\beta &=   2s_{2m} \diff^+ \beta_{2m-1} + 2 \sum_{k=1}^{m-2} s_{2k+2} \diff \beta_{2k+1};\label{F-beta-plus-4m}\\
C_\beta &=2  \sum_{k=1}^{m-1}  (-1)^{m+k+1}s_{2k} \diff^* \beta_{2k+1} \label{C-beta-4m}.
\end{align}
Here, $\diff^+ \beta_{2m-1}$ is the $\Lambda^{2m}_+$-component of $\diff \beta_{2m-1}$. 

\begin{definition}\label{SW4m}
Let $M^{4m}$ be an oriented Riemannian $4m$-manifold with $\Spinc$-structure. The $4m$-dimensional Seiberg--Witten equations on $M$ for $(A,\beta, \phi)$ are:
\begin{align}
D_{A,\beta} (\phi) 
	& = 0, \label{SW4m-Dirac}\\
F_A 
+ 
F^+_\beta
+
C_\beta	
	&=
		q(\phi), \label{SW4m-curvature}
\end{align}
where $D_{A,\beta}$, $F_\beta^+$ and $C_\beta$ are given by~\eqref{Dirac-4mD}, \eqref{F-beta-plus-4m} and~\eqref{C-beta-4m} respectively.
\end{definition} 

Again, the point of our equations  is that there is a Weitzenb\"ock formula for $D_{A,\beta}$ (see~\S\ref{Weitzenbock}) and the principal part of the remainder is exactly $\frac{1}{2}(F_A + F_\beta^++C_\beta)$. Moreover, the equations are elliptic modulo gauge, essentially because the truncated de~Rham complex is elliptic:
\[
\Omega^0 \to \Omega^1 \to \cdots \to \Omega^{2m}_+
\]
(See \S\ref{elliptic} for the details.)

\begin{remark}
In the case $M^{4m} = X^{4m-1} \times \R$ is a Riemannian product, one can consider solutions to the equations which are $\R$-invariant, leading to a system of equations for fields defined purely on $X$.. It turns out that these so-called ``reduced'' equations are equivalent to the $(4m-1)$-dimensional Seiberg--Witten equations on $X$ from Definition~\ref{SW2m+1}. To see this, note first that a 1-form $\alpha$ on $X$ can be made to act on $S_+$ via $c'(\alpha) = c(\alpha)c(\diff t)$ where $c$ is the Clifford action on $M$ and $t$ is the coordinate on $\R$ (with $\diff t$ unit length and positively oriented). In this way we identify $S_+$ as the spin bundle $S \to X$, with $c'$ being the Clifford product on $X$; we also identify $\R$-invariant sections of $S_+$ with $\Gamma(X,S)$. Next note that an $\R$-invariant connection $A$ on $M$ is of the form $A' + if \diff t$ where $A'$ is a connection on $X$ and $f \in \Omega^0(X)$. To recover the Seiberg--Witten equations on $X$, one should take the top degree odd form to be $*f$. Similarly, an $\R$-invariant form $\beta_{2k+1} \in \Omega^{2k+1}(M)$ has the shape $\beta_{2k+1} = \beta'_{2k+1} + \beta'_{2k} \wedge \diff t$ for forms $\beta'$ on $X$. It is the odd-degree forms $\beta'_{2k+1}$, $*\beta'_{2k}$, $*f$ (with appropriate signs) the connection $A$ and spinor $\phi$, which solve the Seiberg--Witten equations on $X$. 
\end{remark}

\subsection{The Seiberg--Witten equations in dimension $4m-2$}\label{SW-4m-2}

This leaves dimension $n=4m-2$ and here things are more complicated. Clifford multiplication gives isomorphisms
\begin{align}
c 
	&\colon 
		i \Lambda^2 \oplus \Lambda^4 \oplus \cdots \oplus s_{2m-2}\Lambda^{2m-2} 
		\to 
		i\su(S_+)
		\label{4m-2-Clifford-isomorphism-plus}\\
c
	&\colon
		s_{2m+2}\Lambda^{2m+2} \oplus \cdots \oplus \Lambda^{4m-4}
		\to
		i\su(S_-)
		\label{4m-2-Clifford-isomorphism-minus}
\end{align}
Notice that we choose to identify even forms of degree $<2m$ with Hermitian endomorphisms of $S_+$ and even forms of degree $>2m$ with hermitian endomorphisms of $S_-$. This is just a convenience; even forms of degree $2m$ also act on $S_-$ for example. 

This time, if we were to work only with $S_+$ say, there is no corresponding bundle of odd degree forms with the correct rank to set up an elliptic system. For example, in dimension~6, $i\su(S_+)$ has rank 15, so the curvature equation will give 15 constraints whilst gauge fixing provides one more. Meanwhile, the connection $A$ accounts for 6 degrees of freedom and so we are left looking for 10 more degrees of freedom, but there is no bundle of forms with this rank which we can use to parametrise Dirac operators. The way out is to use \emph{two} spinors and connections, leading to two Dirac equations and two curvature equations, with everything coupled via the odd degree forms.

We do this as follows. Given $\phi \in S_+$ and $\psi \in S_-$ we write $q(\phi)$ the differential form corresponding to $E_\phi$ under~\eqref{4m-2-Clifford-isomorphism-plus} and $q(\psi)$ for the differential form corresponding to $E_{\psi}$ under~\eqref{4m-2-Clifford-isomorphism-minus}. Notice that the components of $q(\phi)$  have degree $<2m$, whilst those of $q(\psi)$ have degree $>2m$. We consider equations for $(A,B,\beta, \phi,\psi)$ where $A, B \in \A$, $\beta = \beta_3 + \beta_5 + \cdots + \beta_{4m-5}$ where $\beta_k \in \Omega^k$, $\phi \in \Gamma(S_+)$ and $\psi \in \Gamma(S_-)$. 
We write
\begin{align}
D_{A,\beta,+}
	&=
		D_A
		+
		\sum_{k=1}^{m-2}
		\Big(s_{2k+1} c(\beta_{2k+1}) + s_{4m-2k-3}c(*\beta_{2k+1})\Big)
		+
		s_{2m-1}c(\beta_{2m-1}),
			\label{Dirac-plus-4m-2D}\\
D_{B,\beta,-}
	&=
		D_B
		+
		s_{2m-1}c(*\beta_{2m-1})
		+
		\sum_{k=m}^{2m-3} 
		\Big(s_{2k+1}c(\beta_{2k+1}) + s_{4m-2k-3} c(*\beta_{2k+1})\Big).
			\label{Dirac-minus-4m-2D}
\end{align}
We regard $D_{A,\beta,+}$ as acting on sections of $S_+$ and $D_{B,\beta,-}$ as acting on sections of $S_-$. Notice that $\beta_{2m-1}$ is the only part of $\beta$ which appears in both Dirac operators. The terms $\beta_k$ for $k <2m-1$ appear in $D_{A,\beta,+}$ and the terms $\beta_k$ for $k > 2m-1$ appear in $D_{B,\beta,-}$. We now write
\begin{align}
F_{\beta,+}
	&=
		2\sum_{k=1}^{m-2} 
			s_{2k+2} \diff \beta_{2k+1}
		\label{F-beta-plus-4m-2}\\
C_{\beta,+}
	&=
		 (-1)^{m}2\sum_{k=1}^{m-1}
		 	s_{2k} \diff^*\beta_{2k+1}
		\label{C-beta-plus-4m-2}\\
F_{\beta,-}
	&=
		2\sum_{k=m}^{2m-3} 
			s_{4m-2k-4} \diff \beta_{2k+1}
		\label{F-beta-minus-4m-2}\\
C_{\beta,-}
	&=
		(-1)^{m+1} 2\sum_{k=m-1}^{2m-3} 
			s_{4m-2k-2} *(\diff^* \beta_{2k+1})
		\label{C-beta-minus-4m-2}
\end{align}
The $\pm$ in the suffices $F_{\beta,\pm}, C_{\beta,\pm}$ are related to whether the $\beta_k$ involved appear in the Dirac operator acting on positive or negative spinors. Again, $F_{\beta,\pm}$ is essentially comprised of $\diff \beta$ and $C_{\beta,\pm}$ is made from $\diff^*\beta$. 

\begin{definition}\label{SW4m-2}
Let $M^{4m-2}$ be an oriented Riemannian $(4m-2)$-manifold with $\Spinc$ -structure. The $(4m-2)$-dimensional Seiberg--Witten equations on $M$ for $(A,B, \beta, \phi, \psi)$ are 
\begin{align}
D_{A,\beta,+} \phi &= 0,\label{SW4m-2-Dirac-plus}\\
F_{A} + F_{\beta,+} + C_{\beta,+} &=q(\phi),\label{SW4m-2-curvature-plus}\\
D_{B,\beta,-} \psi &= 0,\label{SW4m-2-Dirac-minus}\\
(-1)^{m+1}i(*F_{B}) + F_{\beta,-} + C_{\beta,-} &= q(\psi), \label{SW4m-2-curvature-minus}
\end{align}
where $D_{A,\beta,+}$, $D_{B,\beta,-}$, $F_{\beta,\pm}$ and $C_{\beta,\pm}$ are defined in equations~\eqref{Dirac-plus-4m-2D}--\eqref{C-beta-minus-4m-2}
\end{definition}

Once again, the point is that the Dirac operators $D_{A,\beta,+}$ and $D_{B,\beta,-}$ have Weitzenb\"ock formulae (see \S\ref{Weitzenbock}) in which the principal parts of the remainders are precisely $\frac{1}{2}(F_A+F_{\beta,+}+C_{\beta,+})$ and $\frac{1}{2}((-1)^{m+1} i (*F_B) + F_{\beta,-}+C_{\beta,-})$ respectively.

Since there are two spinors and two connections in play, the appropriate gauge group is now $\G = \Map(M, S^1\times S^1)$ where the first factor acts by pull-back on $(A,\phi)$, the second by pull-back on $(B,\psi)$ and both factors act trivially on $\beta$. This action preserves the above equations. With this in mind, the equations~\eqref{SW4m-2-Dirac-plus}--\eqref{SW4m-2-curvature-minus} are elliptic modulo gauge, something which essentially comes down to ellipticity of the de~Rham complex again. The details are found in \S\ref{elliptic} but we explain here how the numerology works out in dimension 6. The important variables here are the connections $A,B$ and the 3-form $\beta$. (The Dirac equations are already elliptic and so $\phi,\psi$ do not concern us for this discussion.) Each curvature equation is 15 constraints, whilst fixing for both factors in the gauge action brings another 2 constraints, making $15+15+2=32$ in total; meanwhile, each connection gives 6 degrees of freedom (the rank of $\Lambda^1$) and the rank of $\Lambda^3$ is 20, so $(A,B,\beta_3)$ corresponds to $6+6+20=32$ degrees of freedom, which is equal to the number of constraints. 

\begin{remarks}
We make three remarks concerning the $4m-2$-dimensional equations.
\begin{enumerate}
\item
We could equally have taken both spinors to be sections of $S_+$. One reason to take sections of both $S_+$ and $S_-$ is that in dimension~$4m-2$, geometric structures do not single out a ``preferred orientation''. For example, in dimension~4, a symplectic manifold $(M,\omega)$ has a preferred orientation, given by $\omega^2$ and typically $M$ will not have a symplectic structure inducing the opposite orientation. However, if $\dim M = 6$, then $\omega$ and $-\omega$ induce opposite orientations. Since geometric structures do not single out an orientation in dimension $4m-2$, we choose the same behaviour for the Seiberg--Witten equations. This choice affects the index of the equations, as computed in \S\ref{elliptic}, but not the analytic estimates in \S\ref{analysis}.
\item
Another interesting choice is to begin with \emph{two} $\Spinc$-structures $S_\pm$ and $W_{\pm}$. We may then take $A \in \A(S_+)$, $\phi \in \Gamma(S_+)$ whilst $B \in \A(W_-)$, $\psi \in \Gamma(W_-)$. Again, this choice will affect the index of the equations, but not the analytic estimates proved later. This choice is convenient in the second author's construction of examples of solutions of the equations over K\"ahler threefolds \cite{Ghosh}.
\item
When $M^{4m-2} = X \times \R$ is a Riemannian product, one can certainly consider solutions to the equations over $M$ which are invariant in the $\R$ direction. This gives a system of equations on the odd-dimensional manifold $X$. Unlike dimensional reduction from $4m$ to $4m-1$ dimensions, however, this time the resulting equations on $X$ are more complicated than the odd-dimensional Seiberg--Witten equations as in Definition~\ref{SW2m+1}. 

One could also dimensionally reduce the $(4m-1)$-dimensional Seiberg--Witten equations on $M^{4m-2}\times \R$ to obtain a system on $M^{4m-2}$. This also gives a system of equations which is more complicated that the $(4m-2)$-dimensional Seiberg--Witten equations of Definition~\ref{SW4m-2}.
\end{enumerate}
\end{remarks}

\section{5-dimensional examples}\label{5D-examples}

Let $(X,\omega)$ be a K\"ahler surface and $L \to X$ a Hermitian holomorphic line bundle, with Chern connection $B$. In this section we describe solutions to the 5-dimensional Seiberg--Witten equations on the total space of the unit circle bundle $M \subset L$. (The ansatz that we use will ultimately require $L$ to be ample and $X$ to have canonical bundle $K = L^3$.)

To begin, recall that there is a canonical $\Spinc$-structure on the K\"ahler surface $X$, with spinor bundles $S^{\mathrm{can}}_+(X) = \Lambda^{0,0} \oplus \Lambda^{0,2}$ and $S^{\mathrm{can}}_-(X) = \Lambda^{0,1}$. We twist this by $L$ to produce a $\Spinc$-structure $S(X)$ with $S_+(X) = L \oplus \Lambda^{0,2}\otimes L$ and $S_-(X) = \Lambda^{0,1}\otimes L$. We will use the fact that the splitting of $S_+(X)$ is preserved by the Clifford action of the K\"ahler form $\omega$; it acts as $2i$ on $L$ and $-2i$ on $\Lambda^{0,2}\otimes L$. 

We next describe the Riemannian metric on $M \to X$. We denote a point of $M$ by $(x,v)$ where $x \in X$ and $v \in L_x$ with $|v|=1$. The connection $B$ in $L$ gives a horizontal distribution in $M$. The horizontal-vertical splitting of $TM$ has the form:
\begin{equation}
T_{(x,v)}M = H_{(x,v)} \oplus V_{(x,v)} \cong \pi^* T_xX \oplus \{ w \in L_x : \left\langle w,v \right\rangle = 0\}.
\label{splitting-TM}
\end{equation}
Here $\left\langle \cdot, \cdot \right\rangle$ denotes the inner-product on $L_x$. This splitting makes $M$ into a Riemannian manifold: we declare~\eqref{splitting-TM} to be orthogonal, the map $\pi \colon M \to X$ to be a Riemannian submersion, and give each fibre of $\pi$ the metric induced from the corresponding fibre of~$L$. 

The vertical space $V_{(x,v)}$ is spanned by $iv$. We define a 1-form $\alpha$ on $M$ by first projecting $TM$ to $V$ against $H$, and then dividing by the generator $iv$. Equivalently, $i\alpha$ is the principal connection 1-form of $B$ when we consider $M$ as the principal frame bundle of $L$. We use the volume form $\frac{1}{2}\alpha \wedge \pi^*\omega^2$ to orient $M$. 

We define a $\Spinc$-structure on $M$ by pulling back $S(X)$:
\begin{equation}
S(M) = \pi^*S(X) = \pi^* L \oplus \pi^*(\Lambda^{0,1}\otimes L) \oplus \pi^*(\Lambda^{0,2}\otimes L)
\label{splitting-S(M)}
\end{equation}
We also pull back the Hermitian structure of $S(X)$, making $S(M)$ Hermitian and~\eqref{splitting-S(M)} an orthogonal decomposition. The Clifford action of covectors on $M$ works as follows. Under the splitting~\eqref{splitting-TM}, the horizontal covectors in $T_{(x,v)}M$ are identified with covectors in $T^*_xX$. These act on $S(X)$ and so on $S(M)$. Notice that the horizontal covectors swap the summands $\pi^*(L \oplus \Lambda^{0,2}\otimes L)$ and $\pi^*(\Lambda^{0,1} \otimes L)$. The remaining direction is given by the 1-form $\alpha$. We define the Clifford action of $\alpha$ by
\begin{equation}
c(\alpha) = \begin{cases}
	-i \Id & \text{on } \pi^*(L \oplus \Lambda^{0,2}\otimes L)\\
	\phantom{-}i \Id & \text{on } \pi^*(\Lambda^{0,1} \otimes L)
	\end{cases}
\label{Clifford-alpha}
\end{equation}
One checks that, together with the action of horizontal covectors, this satisfies the Clifford relations and makes $S(M)$ into the spin bundle of a $\Spinc$-structure on $M$. (In particular, the signs here are chosen so that $c(\dvol) = i$, as it must in dimension 5.)

The line bundle associated to $S(X)$ is $L^2 \otimes K^{-1}$. Pulling this back gives the line bundle associated to $S(M)$. The bundles $L\to X$ and $K \to X$ both have unitary connections (from the Hermitian metric on $L$ and the K\"ahler metric on $X$). We pull these connections back to $M$ and use them to define a connection on $ \pi^*(L^2 \otimes K^{-1})$, which we denote $A$. This connection will ultimately be part of our solution to the Seiberg--Witten equations.

For the spinor, we begin by noting that there is a tautological section of $\pi^*L \to M$, given by $\tau(x,v) = v$. Since $\pi^*L$ is a summand of $S(M)$ we can regard $\tau$ as a spinor over $M$. Since $\tau$ is a unit-length section of the first summand of the orthogonal decomposition~\eqref{splitting-S(M)}, it is  straight-forward to write the action of $q(\tau)$ with respect to~\eqref{splitting-S(M)}:
\[
c(q(\tau)) 
	= 
		\begin{pmatrix}
		\frac{3}{4} \Id & 0 & 0 \\
		0 & - \frac{1}{4} \Id & 0 \\
		0 & 0 & - \frac{1}{4} \Id 
		\end{pmatrix}
\]
(This is trace zero: the middle $\Id$ acts on the rank~2 bundle $\pi^*(\Lambda^{0,1} \otimes L)$, whilst the $\Id$ operators in the corners act on the line bundles $\pi^*L$ and $\pi^*(\Lambda^{0,2} \otimes L)$.) 

To recognise this in terms of even-degree forms (via~\eqref{Clifford-isomorphism-2m+1}), we give the Clifford  action of $i\pi^*\omega$ and $\pi^*\omega^2$ on $S(M)$. These both preserve the splitting~\eqref{splitting-S(M)} of $S(M)$ and they act on the summands as follows:
\[
c(i\pi^*\omega)
	=
		\begin{pmatrix}
		2 \Id &0&0\\
		0&0&0\\
		0&0&-2\Id
		\end{pmatrix},
\qquad
c(\pi^*\omega^2)
	=
		\begin{pmatrix}
		-2 \Id &0&0\\
		0& 2\Id &0\\
		0&0&-2\Id
		\end{pmatrix}
\]
The formula for $c(i\pi^*\omega)$ follows from the action of $\omega$ on $S_+(X)$; the formula for $c(\pi^*\omega^2)$ follows from the fact that $\omega^2$ acts as multiplication by $\mp2$ on $S_{\pm}(X)$. From here we obtain
\begin{equation}
\label{tautological-q(tau)}
q(\tau) = \frac{1}{8} \left( 2i \pi^* \omega - \pi^* \omega^2\right)
\end{equation}

We next compute the action of the Dirac operator on $\tau$. Recall the connection $A$ on $\pi^*(L^2 \otimes K^{-1})$ given by pulling back the connections on $L$ and $K$. This determines a spin connection $\nabla_A$ on $S(M)$. We will only be interested in the action of $\nabla_A$ on the summand $\pi^*L \leq S(M)$. We claim that $\nabla_A$ makes this sub-bundle parallel and moreover on this sub-bundle it acts as $\pi^*B$ (where $B$ is the Chern connection of $L$). To see this, note that downstairs, the natural spin connection on $S_+^{\text{can}}(X) = \Lambda^{0,0} \oplus \Lambda^{0,2}$ makes the summand $\Lambda^{0,0}$ parallel, where it acts as ordinary differentiation. Tensoring by $(L,B)$ we get a spin connection on $S_+^{\text{can}}(X) \otimes L = L \oplus \Lambda^{0,2}\otimes L$ which makes $L$ parallel, where it acts as differentiation with respect to $B$. Now pulling back to $M$ we see the analogous statement is true for $\nabla_A$ and $\pi^*L$. 

We now compute $D_A \tau$. Choose a local orthonormal frame for $TM$ in which $e_1,e_2,e_3,e_4$ are horizontal and $e_0$ is vertical, with $\alpha(e_0) =1$. By the tautological definition of $\tau$, we have $\nabla^A_{e_0} \tau = \nabla^{\pi^*B} \tau = i\tau$. Meanwhile, $\tau$ is parallel for $\nabla^{\pi^*B}$, and so also for $\nabla_A$, in the horizontal directions. This is because the horizontal distribution in $M$ is determined by the same connection in $L$ which gives rise to $B$ itself. From here we see that $\nabla^A \tau = i\alpha \otimes \tau$ and so, using~\eqref{Clifford-alpha},  $D_A \tau =  i c(\alpha)\tau = \tau$.

We are now ready to solve the 5-dimensional Seiberg--Witten equations. Recall that these equations are for a spinor $\phi$, connection $A \in \A$, and forms $\beta_3 \in \Omega^3$, $\beta_5 \in \Omega^5$; they read
\begin{align*}
\left(D_A + c(\beta_3) - c(*\beta_3) + i  c(*\beta_5) \right) \phi &= 0,\\
F_A - 2 i \diff^* \beta_3 + 2\diff \beta_3 + 2 \diff^* \beta_5 & = q(\phi).
\end{align*}
We take $\phi = 2\sqrt{2} \tau$, for $A$ we use the above choice of connection in $\det S(M)$ and for the odd-degree forms we set $\beta_3 = \frac{1}{2} \alpha \wedge \pi^*\omega$ and $\beta_5 = - \alpha \wedge \omega^2$. Then 
\begin{align*}
c(\beta_3)\phi 
	& = \frac{1}{2}c(\alpha) \circ c(\pi^*\omega) (\phi)
	 = - \phi\\
c(*\beta_3)\phi
	& = \frac{1}{2}c(\pi^*\omega) \phi 
	 = - i \phi\\
c(*\beta_5) \phi
	&= -\phi
\end{align*}
Together with $D_A \phi = \phi$ it follows that the Dirac equation is solved.

Meanwhile, since $*\beta_3 = \frac{1}{2}\pi^*\omega$ we have $\diff^*\beta_3 = 0$, whilst $\diff \beta_3 = \frac{1}{2}\diff \alpha \wedge \pi^*\omega$ and $\diff^*\beta_5=0$. Also $F_A = \pi^*(2 F_B - F_K)$, since $A$ is the pull-back of the corresponding connection in $L^2 \otimes K^{-1}$. So the curvature equation reads
\begin{equation}
 \pi^*(2 F_B -  F_K) + \diff \alpha \wedge \pi^*\omega = 2i\pi^* \omega - \pi^*\omega^2
\label{curvature-equation-circle-bundle}
\end{equation}
The 4-form part of this equation is satisfied if $\diff \alpha = - \pi^*\omega$. Recalling that $i\alpha$ is the connection 1-form of $B$, we have $\diff \alpha = -i \pi^*F_B$ and so we require that $F_B = -i\omega$. In particular, $L$ must be ample, with $c_1(L) = 2\pi[\omega]$.  Now the 2-form part of the curvature equation reads $F_K = -4i \omega$. This is satisfied provided $c_1(X) =  - 4c_1(L)$ and $\omega$ is K\"ahler--Einstein. 

Examples are relatively easy to find: any $X$ with canonical bundle of the form $K = L^4$ for some ample $L$ will do since by Aubin--Yau all such surfaces carry K\"ahler--Einstein metrics. For concrete examples, take a degree $d$ hypersurface in $X \subset \C\P^3$. By adjunction, this has $K = \mathcal{O}(d - 4)|_X$. When $d = 4k$ for $k \geq 2$, we see that $K = L^4$ where $L = \mathcal{O}(k-1)|_X$ is ample.

\section{Ellipticity}\label{elliptic}

\begin{proposition}
The odd-dimesional Seiberg--Witten equations~\eqref{SW2m+1-Dirac} and~\eqref{SW2m+1-curvature} over $M^{2m+1}$ are elliptic modulo gauge, with index zero. 
\label{SW-odd-elliptic}
\end{proposition}

\begin{proof}
The equations~\eqref{SW2m+1-Dirac},~\eqref{SW2m+1-curvature} define a map
\begin{gather*}
\SW 
	\colon 
		\A \times \Omega^{\odd, \geq 3} \times \Gamma(S) 
			\to 
		\left(  i \Omega^2 \oplus \Omega^4 \oplus \cdots \oplus s_{2m}\Omega^{2m} \right)
		\times \Gamma(S) \\
\SW(A,\beta,\phi) = (F_A + F_\beta + C_\beta - q(\phi), D_{A,\beta}\phi)  
\end{gather*}
Suppose $\delta(A,\beta,\phi) = (2ia,b,\sigma)$ is an infinitesimal perturbation of $(A,\beta,\phi)$, where $a\in \Omega^1$, $b \in \Omega^{\odd, \geq 3}$ and $\sigma \in \Gamma(S)$. As $a,b$ vary, $a+b$ fills out the space $\Omega^{\odd}$ of \emph{all} odd degree forms. The linearisation of $\SW$ at $(A,\beta,\phi)$ is
\begin{gather*}
\diff_{(A,\beta,\phi)}\SW 
	\colon 
		\Omega^{\odd} \oplus \Gamma(S) 
			\to 
		 \left(  
			i \Omega^2 \oplus \Omega^4 \oplus \cdots \oplus s_{2m}\Omega^{2m}
		\right)
			\oplus
		\Gamma(S)\\
\diff_{(A,\beta,\phi)} \SW(a+b,\sigma) 
	= 
		\left(2ida + F_b + C_b - \diff_{\phi}q (\sigma), D_{A,\beta}\sigma +  c(ia+b) \phi\right)
\end{gather*}
(Note $F_\beta$ and $C_\beta$ are linear in $\beta$ and so equal to their own derivative.) We supplement this with the Coulomb gauge condition $2 \diff^* \colon \Omega^1 \to \Omega^0$ and discard the zeroth order terms $\diff_\phi q (\psi)$ and $c(a)\phi$ which do not affect ellipticity or the index. This leaves the map
\begin{gather*}
L \colon \Omega^{\odd} \oplus \Gamma(S) \to s\Omega^{\even} \oplus \Gamma(S)\\
L(a,b,\sigma) = \left( 2\diff^*a + 2i \diff a + F_b + C_b, D_A \sigma \right)
\end{gather*}
where $s\Omega^{\even} = \bigoplus_{k=0}^m s_{2k}\Omega^{2k}$. The point is that the first component, $2(\diff^*a + i\diff a) + F_b + C_b$, is essentially the operator $2(\diff + \diff^*)$ acting on $a+b$, just with some extraneous signs and factors of $i$.  This doesn't affect invertibility of the symbol, nor does it change the kernel and cokernel. Since both $\diff +\diff^*$ and $D_A$ are elliptic with index zero, the same is true for~$L$. 
\end{proof}

\begin{proposition}
The $4m$-dimensional Seiberg--Witten equations~\eqref{SW4m-Dirac} and~\eqref{SW4m-curvature} are elliptic modulo gauge, with index:
\[
- 1 + b_1 - b_2 + \cdots + b_{2m-1} - b_{2m}^+ +  2\int_M c_1(L) \wedge \Td(M) 
\]
\end{proposition}

\begin{proof}
The equations~\eqref{SW4m-Dirac} and~\eqref{SW4m-curvature} define a map
\begin{gather*}
\SW \colon \A \times \left( \Omega^{3} \oplus \Omega^5 \oplus \cdots \oplus \Omega^{2m-1}\right) \times \Gamma(S_+)
	\to
		\left(i
		\Omega^2 \oplus \Omega^4 \oplus \cdots \oplus s_{2m} \Omega^{2m}_+
		\right) 
		\times \Gamma(S_+)\\
\SW(A,\beta,\phi)
	=
		\left( 
		F_A + F^+_\beta + C_\beta - q(\phi), D_{A,\beta}\phi
		\right) 
\end{gather*}
Following the same arguments as in the proof of Proposition~\ref{SW-odd-elliptic}, we conclude that the ellipticity of the equations modulo gauge is equivalent to that of the operator
\begin{gather*}
L \colon  \Omega^1 \oplus \Omega^3 \oplus \cdots \oplus \Omega^{2m-1} \oplus \Gamma(S_+)
	\to
		\Omega^0 \oplus i \Omega^2 \oplus \cdots \oplus s_{2m}	\Omega^{2m}_+ 
		\oplus 
		\Gamma(S_-)\\
L(a+b,\sigma) = \left( 2d^*a + 2ida + F_b^+  + C_b, D_A\sigma \right)
\end{gather*}
where $F_b^+$ and $C_b$ are defined by~\eqref{F-beta-plus-4m} and~\eqref{C-beta-4m} respectively. The point now is that the first component, $(2\diff^*a + i\diff a) + F_b^++C_b$ is essentially the elliptic operator corresponding to the truncated de Rham complex
\[
\Omega^0 \stackrel{\diff}{\to} \Omega^1  \stackrel{\diff}{\to} \cdots \to \Omega^{2m-1}
 \stackrel{\diff^+}{\to}\Omega^{2m}_+
\] 
The fact that $L$ has some additional signs and factors of $i$ affects neither the invertibility of the symbol nor the kernel and cokernel. It follows that the equations are elliptic and have the same index as $\ind(D_A) - \sum_{k=0}^{2m-1} (-1)^kb_k - b^+_{2m}$. The result now follows from the Atiyah--Singer index theorem~\cite{Atiyah-Singer}. 
\end{proof}

\begin{proposition}
The $(4m-2)$-dimensional Seiberg--Witten equations~\eqref{SW4m-2-Dirac-plus}--\eqref{SW4m-2-curvature-minus} are elliptic modulo gauge, with index $-\chi(M)$ where $\chi(M)$ denotes the Euler characteristic of $M$.
\end{proposition}
\begin{proof}
To ensure the equations fit nicely on the page, we use the shorthand
\begin{align*}
E_+(A,\beta,\phi) &= F_A + F_{\beta,+} + C_{\beta,+} - q(\phi)\\
E_-(B,\beta,\psi) &= (-1)^{m+1}i (*F_B) + F_{\beta,-} + C_{\beta,-} - q(\psi)
\end{align*}
so that the curvature equations are $E_+ = 0 = E_-$. The equations~\eqref{SW4m-2-Dirac-plus}--\eqref{SW4m-2-curvature-minus} define a map
\begin{gather*}
\SW 
	\colon 
		\A \times \A \times \bigoplus_{k=1}^{2m-3} \Omega^{2k+1}
		\times \Gamma(S_+) \times \Gamma(S_-)
	\to
		\bigoplus_{k=1}^{2m-2} s_{2k}\Omega^{2k}
		\times \Gamma(S_-) \times \Gamma(S_+)\\
\SW(A,B,\beta,\phi,\psi)
	=
		\left( 
			E_+(A,\beta,\phi)
			+
			E_-(B,\beta,\psi),
			D_{A,\beta,+}\phi,
			D_{B,\beta,-}\psi
		\right)
\end{gather*}
The zeros of the map $\SW$ are precisely solutions to~\eqref{SW4m-2-Dirac-plus}--\eqref{SW4m-2-curvature-minus}.

$\A$ is an affine space modelled on $i\Omega^1$. Using the Hodge star, we consider the second copy of $\A$, where $B$ lives, to be modelled on $i \Omega^{4m-3}$. Now an infinitesimal variation $\delta(A,B,\beta,\phi,\psi) = (2ia, 2i\tilde{a}, b, \sigma, \xi)$ has $a \in \Omega^1$, $\tilde{a} \in \Omega^{4m-3}$, $b \in \bigoplus_{k=1}^{2m-3}\Omega^{2k+1}$,  $\sigma \in \Gamma(S_+)$ and $\xi \in \Gamma(S_-)$. The point is that as $a, \tilde{a}$ vary, $a + b + \tilde{a}$ fills out the space $\Omega^{\odd}$ of \emph{all} odd degree forms. 

Recall that the gauge group consists of maps $M \to S^1 \times S^1$, with the first factor acting on $(A,\phi)$ and the second on $(B, \psi)$. So we supplement the linearisation of $\SW$ with two gauge fixing conditions, namely $2\diff^* \colon \Omega^1 \to \Omega^0$, for the Coulomb gauge on $A$ and $-2 \diff \colon \Omega^{4m-3} \to \Omega^{4m-2}$ which corresponds to Coulomb gauge for the gauge action on $B$, together with the Hodge star which we used to identify $T_B \A \cong i\Omega^{4m-3}$. Now following the same arguments as in the proof of Proposition~\ref{SW-odd-elliptic} we see that the ellipticity of the equations modulo gauge is equivalent to that of the operator
\[
L 
	\colon 
		\Omega^{\odd} \oplus \Gamma(S_+) \oplus \Gamma(S_-)
	\to
		s\Omega^{\even} \oplus \Gamma(S_-) \oplus \Gamma(S_+)
\]
where $L(a,b,\tilde{a},\sigma,\xi)$ is equal to
\[		
\left( 
	2\diff^*a 
	+ 
	2i\diff a + F_{b,+} + C_{b,+} 
	+ 
	F_{b,-} + C_{b,-} + 2(-1)^{m+1}i\diff^* \tilde{a} 
	-
	2 \diff \tilde{a},
	D_{A}\sigma, D_{B}\xi
\right)
\]
The point is that now the first component is essentially the operator $2(\diff + \diff^*)$ acting on $a + b + \tilde{a}$, together with some signs and factors of $i$. This doesn't affect the invertibility of the symbol or the index. Since $\diff + \diff^*$, $D_A$ and $D_B$ are all elliptic, the same is true for $L$ and hence the whole system modulo gauge. The index is equal to $\ind(L) = \ind(\diff + \diff^*) + \ind D_A + \ind D_B$ where $ \diff + \diff^*\colon \Omega^{\odd} \to \Omega^{\even}$, $D_A \colon \Gamma(S_+) \to \Gamma(S_-)$  and $D_B \colon \Gamma(S_-) \to \Gamma(S_+)$. The Dirac operators are of opposite chirality and so their indices cancel, whilst the index of $\diff + \diff^*$ is $-\chi(M)$, giving the result.
\end{proof}

\section{Weitzenb\"ock formulae}\label{Weitzenbock}

In this section we compute the Weitzenb\"ock formula for a Dirac operator of the form $D_A + c(\beta)$ where $\beta \in \Omega^*(M,\C)$ is a (possibly inhomogeneous degree) form. To begin we work in the following setting. Let $(M,g)$ be a Riemannian manifold of dimension $n$ and $(E,h) \to M$ a Hermitian vector bundle on $M$. We pick a unitary connection $A$ in $E$. Throughout this section, we will calculate using a local coframe $e_j$ which is stationary at a point $p$ (with respect to the Levi-Civita connection). We write $\nabla_j$ for the corresponding $\nabla^A$-derivative in the direction dual to $e_j$.

\begin{lemma}\label{change-Laplacian}
Let $B \in \Omega^1(\u(E))$ be locally given by $B = \sum_j e_j \otimes B_j$ for sections $B_j$ of $\u(E)$. Then, when acting on sections of $E$, 
\[
\left( \nabla_A + B \right)^* \left( \nabla_A + B\right)
	=
		\nabla_A^*\nabla_A
		-
		2 \sum_j B_j \circ \nabla_j
		-
		\sum_j \left( \left( \nabla_j B_j\right) + B_j^2\right)
\]
\end{lemma}
\begin{proof}
For $\phi \in \Gamma(E)$, 
\begin{align*}
\left( \nabla_A + B \right)^* \left( \nabla_A + B\right) \phi
	&=
		\nabla_A^*\nabla_A \phi
		+
		\nabla_A^*\left( B \phi \right)
		+
		B^* \left( \nabla_A \phi \right)
		+
		B^*B \phi\\
	&=
		\nabla_A^*\nabla_A \phi
		-
		\sum \nabla_j (B_j \phi)
		-
		\sum B_j (\nabla_j \phi)
		-
		\sum B_j^2 \phi\\
	&=
		\nabla_A^* \nabla_A \phi 
		-
		2
		\sum B_j (\nabla_j \phi)
		-
		\sum \left( 
		B_j \left( \nabla_j \phi \right) + B_j^2
		\right)\phi
\qedhere
\end{align*}
\end{proof}
We now consider the case when $(M,g)$ has a $\Spinc$-structure, with $S \to M$ the spin bundle, and a choice of unitary connection $A$ in $L$, determining a spin-connection in $S$, and Dirac operator $D_A$. We pick $\beta \in \Omega^*(M,\C)$ and define 
\[
D_{A,\beta} = D_A + c(\beta) \colon \Gamma(S) \to \Gamma(S)
\]
Meanwhile, let $B \in \Omega^1(\u(S))$ given by 
\[
B = - \frac{1}{2} \sum_j e_j \otimes \left( c(e_j) \circ c(\beta) + c(\beta)^* \circ c(e_j)\right)
\]
Whilst $B$ is at first sight only defined locally, one checks that this definition is independent of the choice of orthonormal coframe $e_j$ and so gives a globally defined section $B \in \Omega^1(\u(S))$. From here we define a unitary connection in $S$ by 
\[
\nabla_{A,\beta} = \nabla_A + B
\]
\begin{proposition}\label{Weitzenbock-proposition}
The connection $\nabla_{A,\beta}$ and Dirac operator $D_{A,\beta}$ satisfy the following Weitzenb\"ock formula:
\begin{align*}
D_{A,\beta}^* D_{A,\beta} - \nabla_{A,\beta}^* \nabla_{A,\beta}
	&=
		\frac{s}{4} + \frac{1}{2} c(F_A)\\
	&\phantom{=}
		\quad	
		+
		\frac{1}{2} \sum \left( 
			c(e_j)  \circ c(\nabla_j \beta) - c(\nabla_j\beta)^* \circ c(e_j)
			\right)\\
	&\phantom{=}
		\qquad		
		+
		\frac{1}{4} \sum \left( 
			c(e_j) \circ c(\beta) + c(\beta)^* \circ c(e_j)
		\right)^2
		+ 
		c(\beta)^* \circ c(\beta)
\end{align*}
where $s$ is the scalar curvature of $(M,g)$. 
\end{proposition}
\begin{remark}\label{Weitzenbock-remarks}
~\begin{enumerate}
\item
At first sight, the two summations on right-hand side are only locally defined, but one can check that they are each independent of the choice of local orthonormal frame and so are actually globally defined expressions. We use this fact in the proof to compute at the point $p$ in which the frame $e_j$ has vanishing derivatives. 
\item
The first line in on the right-hand side contains the remainder terms which don't depend on $\beta$, namely those one sees in the ``standard'' Weitzenb\"ock formula; the second line contains terms which are linear in the first derivatives of $\beta$, this is the ``principal part'' of the remainder (as far as $\beta$ is concerned); the third line contains the terms which are non-linear in $\beta$, they are quadratic and depend only algebraically on $\beta$. 
\item
The first formula of this type, where $\beta \in \Omega^3$, was proved by Bismut~\cite{Bismut}. 
\end{enumerate}
\end{remark}

\begin{proof}[Proof of Proposition~\ref{Weitzenbock-proposition}]
The proof is a direct calculation:
\begin{align*}
D_{A,\beta}^*D_{A,\beta} 
	&= 
		D^2_A+ D_A \circ c(\beta) + c(\beta)^* \circ D_A + c(\beta)^* \circ c(\beta)\\
	&=
		\nabla_A^*\nabla_A + \frac{s}{4} + \frac{1}{2} c(F_A) \\
	&\phantom{=}
		\quad
		+
		\sum c(e_j) \circ \left(c(\nabla_j \beta) + c(\beta) \circ \nabla_j\right)
		+
		\sum c(\beta)^* \circ c(e_j) \circ \nabla_j
		+
		c(\beta)^* \circ c(\beta)
\end{align*}
where we have used the standard Weitzenb\"ock formula for $D_A^2$ in the second line. 

We now isolate the terms in this which are first-order:
\begin{align*}
D^*_{A,\beta}D_{A,\beta}
	&=
		\nabla_A^*\nabla_A + \frac{s}{4} + \frac{1}{2} c(F_A) \\
	&\phantom{=}
		\quad	
		+
		\sum \left[ c(e_j) \circ c(\beta) + c(\beta)^* \circ c(e_j) \right]\circ \nabla_j\\
	&\phantom{=}
		\qquad
		+
		\sum c(e_j) \circ c(\nabla_j \beta) + c(\beta)^* + c(\beta)
\end{align*}
The coefficient of $\nabla_j$ is
\[
\sum \left[ c(e_j) \circ c(\beta) + c(\beta)^* \circ c(e_j) \right]
\]
Comparing this with the formula from Lemma~\ref{change-Laplacian}, we recognise this as the first-order term in the rough Laplacian created by adding 
\[
B = -\frac{1}{2} \sum e_j \otimes \left( c(e_j) \otimes c(\beta) + c(\beta)^*\circ c(e_j)\right)
\]
to the connection $A$. This explains the definition of $\nabla_{A,\beta} = \nabla_A + B$. From here Lemma~\ref{change-Laplacian} gives
\begin{align*}
D^*_{A,\beta}D_{A,\beta}
	&=
		\nabla_{A,\beta}^* \nabla_{A,\beta} + \frac{s}{4} + \frac{1}{2} c(F_A)\\
	&\phantom{=}
		\quad	
		+
		\sum \nabla_j B_j
		+
		\sum B_j^2
		+
		\sum c(e_j) \circ c(\nabla_j \beta)
		+
		c(\beta)^* c(\beta)		
\end{align*}
Substituting $B_j = -\frac{1}{2}(c(e_j)\otimes c(\beta) + c(\beta)^*\circ c(e_j))$ and using the fact that $\nabla e_j = 0$ at $p$ (and that Clifford multiplication is parallel) we obtain the stated formula. This holds at $p$ but, as mentioned in Remark~\ref{Weitzenbock-remarks}, this is sufficient to prove the formula globally. 
\end{proof}

We will relate the terms involving $\nabla_j \beta$ to $c(\diff \beta)$ and $c(\diff^*\beta)$. To do this we need the following simple lemmas. Both are well known, but we include the short proofs for completeness. 

\begin{lemma}\label{Clifford-exterior-derivative}
For $\gamma \in \Omega^k$, 
\begin{align*}
c(\diff \gamma) 
	&= 
		\frac{1}{2} \sum \left( 
			c(e_j) \circ c(\nabla_j \gamma) + (-1)^{k} c(\nabla_j \gamma) \circ c(e_j)
		\right)\\
c(\diff^* \gamma)
	&=
		\frac{1}{2} \sum \left( 
			c(e_j) \circ c(\nabla_j \gamma) - (-1)^{k} c(\nabla_j \gamma) \circ c(e_j)
		\right)
\end{align*}
\end{lemma}
\begin{proof}
We begin by recalling two consequences of the  Clifford relations. For $\eta \in \Omega^1$ and $\gamma \in \Omega^k$, 
\begin{align*}
c( \eta \lrcorner \gamma)
	&=
		-\frac{1}{2} \left( 
			c(\eta) \circ c(\gamma) + (-1)^{k+1} c(\gamma) \circ c(\eta)
		\right),\\
c(\eta \wedge \gamma) 
	&= 
		\frac{1}{2} \left( 
		c(\eta ) \circ c(\gamma) + (-1)^k c(\gamma) \circ c(\eta)
		\right)
\end{align*}
In the first formula, $\eta \lrcorner \gamma \in \Omega^{k-1}$ is the contraction of $\eta$ and $\gamma$ defined via the metric: $\eta \lrcorner \gamma = \iota_{\eta^\#} (\gamma)$  where $\eta^\#$ is the vector metric-dual to $\eta$. 

The Lemma now follows from these formulae and the fact that $\diff \gamma = \sum e_j \wedge \nabla_j \gamma$ and $\diff^* \gamma =-  \sum e_j \lrcorner \nabla_j \gamma$. 
\end{proof}

\begin{lemma}\label{Clifford-Hodge}
Let $\gamma \in \Omega^k$. Then 
\[
c(*\gamma) 
	=  
		\begin{cases}
		\pm(-1)^{m+k(k+1)/2} i^m c(\gamma) 
			& \text{on }S_{\pm} \text{ if } \dim(M) = 2m\\
		(-1)^{m+1 + k(k+1)/2} i^{m+1} c(\gamma)
			& \text{on all of } S \text{ if } \dim(M) = 2m+1
		\end{cases}
\]
\end{lemma}
\begin{proof}
By linearity it suffices to prove this for $\gamma = e_1 \wedge \cdots \wedge e_k$ where $e_1, \ldots, e_n$ are an oriented orthonormal frame. For this choice of $\gamma$ we have $c(\dvol) = c(\gamma \wedge *\gamma) = c(\gamma) \circ c(*\gamma)$. Meanwhile, $c(\gamma)^2 = (-1)^{k(k+1)/2}$. So $c(*\gamma) = (-1)^{k(k+1)/2}c(\gamma)\circ c(\dvol)$. Now
\[
c(\dvol) 
	= 
		\begin{cases}
		\pm (-1)^m i^m& \text{on }S_{\pm} \text{ if } \dim(M) = 2m\\
		(-1)^{m+1} i^{m+1} & \text{on all of } S \text{ if } \dim(M) = 2m+1
		\end{cases}
\]
which gives the result.
\end{proof}

\begin{remark}\label{self-dual-action}
On a manifold of dimension $4m$, when $\gamma \in \Omega^{2m}$ has middle degree, we see from Lemma~\ref{Clifford-Hodge} that $c(*\gamma) = \pm c(\gamma)$ on $S_{\pm}$. So on $S_\pm$, $c(\gamma) = c(\gamma^{\pm})$.
\end{remark}

We are now in a position to prove that the curvature equations in our Seiberg--Witten equations do indeed prescribe the principal parts of the various Weitzenb\"ock identities.

\subsection{Odd dimensions}

In dimension~$2m+1$ we use $D_{A,\beta} = D_{A} + c(\beta)$ where $\beta$ has the form
\begin{equation}
\beta 
	= 
	\sum_{k=1}^{m-1} 
		s_{2k+1} \beta_{2k+1} 
	+ 
	\sum_{k=1}^m i s_{2m-2k} * \beta_{2k+1}
\label{odd-dimensions-beta}
\end{equation}
where the components $\beta_j \in \Omega^j$ are real-valued odd-degree forms. 

\begin{proposition}\label{odd-dim-Weitzenbock-proposition}
\[
D_{A,\beta}^* D_{A,\beta}
=
\nabla_{A,\beta}^*\nabla_{A,\beta}
+
\frac{s}{4}
+
\frac{1}{2} c(F_A + F_\beta + C_\beta)
+ 
Q(\beta)
\]
where $F_\beta$, $C_\beta$ are given by \eqref{F-beta-2m+1} and \eqref{C-beta-2m+1} respectively, and $Q(\beta)$ is a quadratic algebraic expression in $\beta$. 
\end{proposition}
\begin{proof}
By Proposition~\ref{Weitzenbock-proposition}, we must show that 
\begin{equation}
\sum_j \left(c(e_j) \circ c(\nabla_j \beta) - c(\nabla_j \beta)^* \circ c(e_j) \right)
=
c\left( F_\beta + C_\beta \right)
\label{what-we-need-to-show}
\end{equation}
When we substitute~\eqref{odd-dimensions-beta} in the left-hand side of~\eqref{what-we-need-to-show}, and in particular in $c(\nabla_j \beta)^*$, we will have two types of term for which we will have to take the adjoint. Terms of type 1 have the form 
\[
c\left(  s_{2k+1} \nabla_j \beta_{2k+1} \right), \quad k=1,\ldots, m-1
\] 
whilst terms of type 2 have the form 
\[
c\left(is_{2m-2k}\nabla_j (*\beta_{2k+1})\right), \quad k=1,\ldots, m-1.
\] 
Now, by definition of $s_k$, for any form $\gamma$ of degree $k$, $c(s_k \gamma)$ is self-adjoint. So the type 1 terms are self-adjoint, whilst type 2 are skew-adjoint. Similarly the factor in $is_{2m+1}\beta_{2m+1}$ ensures that the terms $c(is_{2m+1}\nabla_j \beta_{2m+1})$ are also skew-adjoint. With this observation, we see that the left-hand side of~\eqref{what-we-need-to-show} is
\begin{multline*}
\sum_{k=1}^{m-1} s_{2k+1}
	\sum_{j} \left( 
		c(e_j) \circ c(\nabla_j \beta_{2k+1})
		-
		c(\nabla_j \beta_{2k+1}) \circ c(e_j)
	\right)\\
+
i \sum_{k=1}^{m} s_{2m-2k} 
	\sum_j \left( 
		c(e_j) \circ c(\nabla_j (*\beta_{2k+1}))
		+
		c(\nabla_j (* \beta_{2k+1})) \circ c(e_j)
	\right)
\end{multline*}
Applying Lemma~\ref{Clifford-exterior-derivative} we see that the left-hand side of~\eqref{what-we-need-to-show} is
\begin{equation}
2\sum_{k=1}^{m-1} s_{2k+1} c(\diff \beta_{2k+1})
+
2i \sum_{k=1}^{m} s_{2m-2k} c(\diff (*\beta_{2k+1}))
\label{intermediate-check-Weitzenbock}
\end{equation}
By definition of $s_{k}$, $s_{2k+1} = s_{2k+2}$ for all $k$ and so the first term of~\eqref{intermediate-check-Weitzenbock} is exactly $F_\beta$. Meanwhile, on an odd-dimensional manifold,  $\diff^* = (-1)^k*\diff ~*$ when acting on $k$-forms, whilst $*^2 = 1$ in every degree. So $d (*\beta_{2k+1})= - * \diff^*\beta_{2k+1}$. By Lemma~\ref{Clifford-Hodge},
\begin{align*}
i s_{2m-2k}c(d(*\beta_{2k+1})) 
	&= 
		- i s_{2m-2k} c(* \diff^*\beta_{2k+1}) \\
	&= 
		(-1)^{m+1+k(2k+1)}i^{m}s_{2m-2k} c(\diff^*\beta_{2k+1}) \\\
	&=
		(-1)^{m+k+1}i^{m}s_{2m-2k} c(\diff^*\beta_{2k+1})
\end{align*}
Now for all possible values of $m,k$, we have $i^ms_{2m-2k} = (-1)^{km+ \frac{m(m+1)}{2}} s_{2k}$. This means that the second term of~\eqref{intermediate-check-Weitzenbock} is equal to 
\[
2\sum_{k=1}^{m}
	(-1)^{km + m+k + 1 + \frac{m(m+1)}{2}} s_{2k} c(\diff^* \beta_{2k+1}) = c(C_\beta)
\]
as claimed. 
\end{proof}

\subsection{Dimension $4m$}

In dimension $4m$, we use $D_{A,\beta} = D_{A} + c(\beta)$ on $S_+$ where $\beta$ has the form
\begin{equation}
\beta 
	= 
\sum_{k=1}^{m-1} \Big( s_{2k+1} \beta_{2k+1} + s_{4m-2k-1} *\beta_{2k+1} )\Big)
\label{beta-4m}
\end{equation}
where the components $\beta_k \in \Omega^j$ are real-valued odd-degree forms. 

\begin{proposition} \label{4m-dim-Weitzenbock-proposition}
When acting on sections of $S_+$, 
\[
D^*_{A,\beta} D_{A,\beta}
=
\nabla^*_{A,\beta} \nabla_{A,\beta}
+
\frac{s}{4}
+ \frac{1}{2} c \left( 
F_A + F_\beta^+ + C_\beta
\right)
+Q(\beta)
\]
where $F_{\beta}^+$ and $C_{\beta}$ are given by~\eqref{F-beta-plus-4m}  and \eqref{C-beta-4m} respectively and $Q(\beta)$ is a quadratic algebraic expression in $\beta$. 
\end{proposition}

\begin{proof}
By Proposition~\ref{Weitzenbock-proposition} we must show that 
\begin{equation}
\sum_j \left(c(e_j) \circ c(\nabla_j \beta) - c(\nabla_j \beta)^* \circ c(e_j) \right)
=
c\left( F^+_\beta + C_\beta \right)
\label{what-we-need-to-show2}
\end{equation}
When we substitute~\eqref{beta-4m} in to the left-hand side of~\eqref{what-we-need-to-show2} we encounter two types of term:
\[
s_{2k+1} c\left(\nabla_j \beta_{2k+1}\right),
\quad \text{and} \quad
s_{4m-2k-1}c \left( \nabla_j (*\beta_{2k+1})\right)
\]
By definition of $s_{j}$ both of these are self-adjoint. So the left-hand side of~\eqref{what-we-need-to-show2} is
\begin{multline*}
\sum_{k=1}^{m-1} s_{2k+1}
	\sum_j \left( 
		c(e_j) \circ c(\nabla_j \beta_{2k+1}) 
			- 
		c(\nabla_j \beta_{2k+1}) \circ c(e_j)
	\right)\\
+ \sum_{k=1}^{m-1} s_{4m-2k-1}
	\sum_j \left(
		c(e_j) \circ c(\nabla_j (*\beta_{2k+1}))
			-
		c(\nabla_j (*\beta_{2k+1})) \circ c(e_j) 
	\right)
\end{multline*}
By Lemma~\ref{Clifford-exterior-derivative} we see that the left-hand side of~\eqref{what-we-need-to-show2} is
\[
2 \sum_{k=1}^{m-1} \left( 
s_{2k+1} c(\diff \beta_{2k+1}) + s_{4m-2k-1} c(\diff (*\beta_{2k+1})
\right)
\]
On an even-dimensional manifold, $\diff ^* = - * \diff~*$ in every degree, whilst in even degrees, $*^2 = 1$. So $\diff (* \beta_{2k+1}) = - * \diff^* \beta_{2k+1}$. Now by Lemma~\ref{Clifford-Hodge}, we have
\begin{align*}
c(\diff (*\beta_{2k+1}))
	&=
		- c(* \diff^* \beta_{2k+1})\\
	&=
		(-1)^{m+1+k(2k+1)}c(\diff^*\beta_{2k+1})\\
	&=
		(-1)^{m+1+k} c(\diff^*\beta_{2k+1})
\end{align*}
(We replace $m$ in Lemma~\ref{Clifford-Hodge} by $2m$ since for us $n=4m$, and $k$ by $2k$ since we take $*$ of the $2k$-form $\diff^*\beta_{2k+1}$.) One now checks that $s_{2k+1} = s_{2k+2}$ and $s_{4m-2k-1}=s_{4m-2k}=s_{2k}$. Putting the pieces together we see that the left-hand side of~\eqref{what-we-need-to-show2} is
\[
\sum_{k=1}^{m-1} s_{2k+2} c(\diff \beta_{2k+1}) + (-1)^{m+1+k} s_{4m-2k}c(\diff^* \beta_{2k+1})
\]
Finally,  for a form $\gamma$ of degree $2m$ acting on $S_+$, $c(\gamma) =c(\gamma^+)$ (see Remark~\ref{self-dual-action}). So the left-hand side of~\eqref{what-we-need-to-show} is equal to $c(F_\beta^+ + C_\beta)$ as claimed. 
\end{proof} 

\subsection{Dimension $4m-2$}

In dimension $4m-2$ we consider two Dirac operators. The first acts on $\Gamma(S_+)$ and has the form $D_{A,\beta,+} = D_A + c(\beta)$ where
\begin{equation}
\beta
	= 
		\sum_{k=1}^{m-2} 
			s_{2k+1} \beta_{2k+1} 
		+ 
		\sum_{k=1}^{m-1}
			s_{4m-2k-3} *\beta_{2k+1}
\label{beta-plus}
\end{equation}
The second acts on $\Gamma(S_-)$ and has the form $D_{B,\beta,-} = D_B + c(\beta)$ where 
\begin{equation}
\beta 
	=
		\sum_{k=m-1}^{2m-3} 
			s_{2k+1} \beta_{2k+1} 
		+
		\sum_{k=m}^{2m-3}
			s_{4m-2k-3} *\beta_{2k+1}		
\label{beta-minus}
\end{equation}
In both cases the $\beta_j \in \Omega^j$ are odd-degree forms. 

\begin{proposition}\label{4m-2-dim-Weitzenbock-proposition}
~\begin{enumerate}
\item
When acting on sections of $S_+$, with $\beta$ as in~\eqref{beta-plus}, 
\[
D^*_{A,\beta,+} D_{A,\beta,+}
	=
		\nabla_{A,\beta,+}^*\nabla_{A,\beta,+}
		+
		\frac{s}{4}
		+
		\frac{1}{2} c \left( 
		F_A + F_{\beta,+} + C_{\beta,+}
		\right)
		+
		Q(\beta)
\]
where $F_\beta$ and $C_\beta$ are given by~\eqref{F-beta-plus-4m-2} and \eqref{C-beta-plus-4m-2} respectively, and $Q(\beta)$ is a quadratic algebraic expression in $\beta$. 
\item
When acting on sections of $S_-$, with $\beta$ as in~\eqref{beta-minus}, 
\[
D^*_{B,\beta,-} D_{B,\beta,-}
	=
		\nabla_{B,\beta,-}^*\nabla_{B,\beta,-}
		+
		\frac{s}{4}
		+
		\frac{1}{2} c \left( 
		(-1)^{m+1}i(*F_B) + F_{\beta,-} + C_{\beta,-}
		\right) 
		+
		Q(\beta)
\]
where $F_{\beta,-}$ and $C_{\beta,-}$ are given by~\eqref{F-beta-plus-4m-2} and \eqref{C-beta-plus-4m-2} respectively, and $Q(\beta)$ is a quadratic algebraic expression in $\beta$. 
\end{enumerate}
\end{proposition}

\begin{proof}
We begin with the proof of the first equation. By Proposition~\ref{Weitzenbock-proposition} we must show that when $\beta$ is given by~\eqref{beta-plus}, and we are acting on $S_+$,
\begin{equation}
\sum_j \left( 
	c(e_j) \circ c(\nabla_j \beta) - c(\nabla_j \beta)^* \circ c(e_j)
\right)
	=
		c \left( 
		F_{\beta,+} + C_{\beta,+}
		\right)
\label{what-we-need-to-show3}
\end{equation}
Substituting~\eqref{beta-plus} in the left-hand side of`\eqref{what-we-need-to-show3}, we encounter two types of term:
\[
s_{2k+1}c(\nabla_j \beta_{2k+1})
	\quad \text{and} \quad
s_{4m-2k-3}c( \nabla_j(* \beta_{2k_1}))
\]
In both cases, by definition of $s_j$, they are self-adjoint. As in the case of dimension~$4m$, we then see from Lemma~\ref{Clifford-exterior-derivative} that the left-hand side of~\eqref{what-we-need-to-show3} is equal to 
\begin{equation}
2 \sum_{k=1}^{m-2} s_{2k+1} c(\diff \beta_{2k+1})
+
2 \sum_{k=1}^{m-1} s_{4m-2k-3} c(\diff(* \beta_{2k+1}))
\label{two-sums}
\end{equation}

The first sum here is equal to $F_{\beta,+}$. For the second note that, on even-dimensional manifold, $\diff^* = - * \diff~*$ in every degree, whilst in even degrees $*^2=1$. So $\diff (*\beta_{2k+1}) = - * \diff^*\beta_{2k+1}$. Next, by Lemma~\ref{Clifford-Hodge}, we have
\[
c(d(* \beta_{2k+1}))
	=
		- c(*(\diff^*\beta_{2k+1)}))
	=
		(-1)^{2m+1+k(2k+1)}i^{2m-1} c(\diff^*\beta_{2k+1})
\]
We now use the fact that
\[
s_{4m-2k-3} = s_{2k+1} = (-1)^{k}i s_{2k}
\]
to obtain
\[
s_{4m-2k-2} c(\diff(*\beta_{2k+1}))
	=
		(-1)^m s_{2k} c(\diff^*\beta_{2k+1})
\]
So the second sum in~\eqref{two-sums} is $C_{\beta,+}$, which completes the proof of the first part of the statement.  

The proof of the second part is almost identical. The additional factor of $-1$ in~\eqref{C-beta-plus-4m-2} appears because we are acting on the bundle of negative spinors and so there is an extra factor of $-1$ when we apply Lemma~\ref{Clifford-Hodge}.
\end{proof}

\subsection{Zeroth-order terms}

In this section we compute the zeroth order term $Q(\beta)$ in the Weitzenb\"ock remainders, in some specific dimensions. This term is crucial, especially its sign,  when one applies the maximum principle to deduce a priori estimates. It is here that the difference in difficulty between $n=3,4$ and $n>4$ becomes evident.

By Proposition~\ref{Weitzenbock-proposition} the term in question is
\begin{equation}
Q(\beta) 
	= 
		c(\beta)^* \circ c(\beta) 
		+  
		\frac{1}{4} \sum_j\left( c(e_j) \circ c(\beta) + c(\beta)^* \circ c(e_j) \right)^2
\label{Q}
\end{equation}
where the precise form of $\beta$ will depend on the dimension under consideration. The first term is non-negative, whilst the second is non-positive. This is because $c(e_j)^* = - c(e_j)$ and so each term in the second sum is a square $S_j^2$ where $S_j = c(e_j) \circ c(\beta) + c(\beta^*)\circ c(e_j)$ is skew-adjoint. This means that there is a competition between these two terms and it is this that makes it impossible in higher dimensions to give a sign to $Q(\beta)$.

In dimension~3 (a case that is already in the literature of course), we consider $D_A + c(\beta)$ where $\beta = i * \beta_3$ where $\beta_3 \in \Omega^3$.  In this case, $c(\beta)$ is multiplication by the function $-i *\beta_3$. It follows that $Q(\beta) = |\beta_3|^2$. In particular it is non-negative. Dimension~4 is even more straightforward since $\beta=0$, and so $Q(\beta)=0$. 

Dimension~5 is the first in which the calculation of $Q(\beta)$ is somewhat involved. We consider $D_A + c(\beta)$ where $\beta = \beta_3 - *\beta_3 + i* \beta_5$  for $\beta_3 \in \Omega^3$ and $\beta_5 \in \Omega^5$. The next lemma shows how $Q(\beta)$ can be expressed as the difference of non-negative terms, given purely in terms of $\beta_3$ and $\beta_5$. 

\begin{lemma}\label{zeroth-order-5D}
On a 5-manifold, with $\beta = \beta_3 - *\beta_3 + i* \beta_5$ we have 
\[
Q(\beta) = c(\beta_3 - *\beta_5)^2 - 4|\beta_3|^2
\]
\end{lemma}

\begin{proof}
By Lemma~\ref{Clifford-Hodge}, $c(*\beta_3) = i c(\beta_3)$ and so, since $c(\beta_3)^* = c(\beta_3)$, we see that
\begin{align*}
c(\beta) 
	&=
		(1-i) c(\beta_3) + i *\beta_5\\
c(\beta)^*
	&=
		(1+i) c(\beta_3) - i * \beta_5\\
c(\beta)^* \circ c(\beta)
	&=
		\left( (1+i) c(\beta_3) - i *\beta_5\right)
		\left( (1-i) c(\beta_3) + i * \beta_5 \right) \\
	&=
		2 c(\beta_3)^2 - 2 (* \beta_5) c(\beta_3) + |\beta_5|^2\\
	&=
		c(\beta) \circ c(\beta)^*
\end{align*}

In what follows will also need two facts:
\begin{align}
\sum c(e_j) \circ c(\beta_3)^2 \circ c(e_j)
	&=
		3 c(\beta_3)^2 - 8 |\beta_3|^2
		\label{conjugate-beta3-squared}\\
\sum c(e_j) \circ c(\beta_3) \circ c(e_j)
	&=
		- c(\beta_3)
		\label{conjugate-beta3}
\end{align}
These kinds of identities are standard in Clifford algebra, but we give the proofs below, both for completeness and also for lack of a specific reference. 

Assuming them for now, however, we can conclude:
\begin{align}
Q(\beta)
	&=
		c(\beta)^*\circ c(\beta) 
		+ 
		\frac{1}{4} \sum_j \left(
			c(e_j) \circ c(\beta) + c(\beta)^*\circ c(e_j)
		\right)^2
		\nonumber\\
	&=
		c(\beta)^*\circ c(\beta) 
		+
		\frac{1}{4}  c(\beta)^* \circ\left( \sum_j c(e_j)^2 \right)\circ c(\beta)
		\label{Q1}\\
	&\phantom{=}
		\quad
		+ 
		\frac{1}{4}\sum_j c(e_j) \circ \left( 
				2 c(\beta_3)^2 - 2 (*\beta_5) c(\beta_3) + |\beta_5|^2
			\right) \circ c(e_j)
		\label{Q2}\\
	&\phantom{=}
		\quad
			+
			\frac{1}{4}\left(\sum_j c(e_j) \circ \left(
				(1-i)c(\beta_3) + i (*\beta_5)
				\right)
				\circ c(e_i)
			\right)\circ  \left(
					(1-i)c(\beta_3) + i (*\beta_5)
				\right)
		\label{Q3}\\
	&\phantom{=}
		\quad
			+
			\frac{1}{4}\left( 
					(1+i) c(\beta_3) - i (*\beta_5) 
				\right)
					 \circ \left(
						\sum_j c(e_j) \circ \left( 
							(1+i) c(\beta_3) - i (*\beta_5) 
						\right) \circ c(e_j)
					\right)
			\label{Q4}
\end{align}
Since $c(e_j)^2 = -1$, the whole first line~\eqref{Q1} is equal to 
\[
-\frac{1}{4} c(\beta)^* \circ c(\beta)
	=
		-\frac{1}{2} c(\beta_3)^2 
		+ \frac{1}{2} (*\beta_5) c(\beta_3) 
		- \frac{1}{4} |\beta_5|^2
\]
In the second line~\eqref{Q2} we apply both~\eqref{conjugate-beta3-squared} and~\eqref{conjugate-beta3} as well as $c(e_j)^2=-1$. This shows that~\eqref{Q2} is equal to:
\[
\frac{3}{2} c(\beta_3)^2 - 4 |\beta_3|^2
+
\frac{1}{2} (*\beta_5)c(\beta_3)
-
\frac{5}{4} |\beta_5|^2
\]
In the next line~\eqref{Q3} we apply~\eqref{conjugate-beta3} to show that the contribution here is equal to:
\[
\frac{1}{4}\left( 
	-(1-i)c(\beta_3) - 5i(*\beta_5)
\right)
\left( 
	(1-i) c(\beta_3) + i (*\beta_5)
\right)
	=
		\frac{i}{2}c(\beta_3)^2 
		- 
		\frac{3}{2}(1+i)(*\beta_5)c(\beta_3)  
		+ 
		\frac{5}{4}|\beta_5|^2
\]
Similarly, we use~\eqref{conjugate-beta3} to show that the overall contribution of the final line~\eqref{Q4} is equal to
\begin{multline*}
\frac{1}{4}\left( 
		(1+i) c(\beta_3) - i (*\beta_5) 
	\right)
	\left( 
	-(1+i)c(\beta_3) + 5i (*\beta_5)
	\right)\\
	=
	-\frac{i}{2}c(\beta_3)^2
	-
	\frac{3}{2}(1-i)(*\beta_5)c(\beta_3)
	+ 
	\frac{5}{4} |\beta_5|^2
\end{multline*}
We now combine the pieces to see that 
\[
Q(\beta) = c(\beta_3)^2 - 2(*\beta_5) c(\beta_3) +|\beta_5|^2 - 4|\beta_3|^2
\]
which simplifies to give the claimed expression.

We complete the proof by showing that~\eqref{conjugate-beta3-squared} and~\eqref{conjugate-beta3} hold. To prove~\eqref{conjugate-beta3-squared}, we first write $\beta_3$ in a local frame as
\[
\beta_3 = \sum_{i_1 < i_2 < i_3} \beta_{i_1,i_2,i_3} e_{i_1 i_2 i_3}
\]
where $e_{i_1i_2i_3}=e_{i_1} \wedge e_{i_2} \wedge e_{i_3}$. Notice that 
\[
c(e_{i_1i_2i_3}) = c(e_{i_1}) \circ c(e_{i_2}) \circ c(e_{i_3})
\] 
In computing $c(\beta_3)^2$ itself, we will encounter 6-fold compositions of the $c(e_i)$. We can then use the Clifford relations  $c(e_i)c(e_j) + c(e_j)c(e_i) = -2 \delta_{ij}$ to permute them and cancel out various terms when adjacent indices are equal.  The end result for each term will be of the form $c(e_{i_1 \ldots i_k})$ for some basis $k$-form $e_{i_1 \ldots i_k}$. Moreover, since terms cancel in pairs, $k$ will be even. Now, note that $c(\beta_3)$ is self-adjoint and so the same is true for $c(\beta_3)^2$. Meanwhile, real 2-forms have skew-adjoint Clifford action. This means that when the above procedure is carried out there are no terms of the form $c(e_{ij})$ for $i<j$ at all. The upshot is that $c(\beta_3)^2 = f + c(\theta)$ where $f$ is real-valued function and $\theta$ is a 4-form. 

The function $f$ comes from products of the form
\[
c(e_{i_1i_2i_3})c(e_{i_1i_2i_3}) =  1
\]
from which it follows that $f = |\beta_3|^2$. For this term we observe that
\[
\sum_j c(e_j) \circ c(f) \circ c(e_j) = - 5|\beta_3|^2
\]

Meanwhile, for $c(\theta)$, when we pre- and post-multiply by $c(e_j)$ we encounter terms of the form
\[
c(e_j) \circ c(e_{i_1i_2i_3i_4}) \circ c(e_j)
\]
where $i_1<i_2<i_3<i_4$. There are two possibilities: either $j \in \{i_1,i_2,i_3,i_4\}$ or all the indices are distinct. In the first case, the Clifford relations give
\[
c(e_j) \circ c(e_{i_1i_2i_3i_4}) \circ c(e_j)
=
 c(e_{i_1i_2i_3i_4})
\] 
In the second case, meanwhile,
\[
c(e_j) \circ c(e_{i_1i_2i_3i_4}) \circ c(e_j)
=
- c(e_{i_1i_2i_3i_4})
\]
For each choice of $i_1,i_2,i_3,i_4$ the first case occurs for 4 values of $j$ and the second just once. So 
\[
\sum_j c(e_j) \circ c(e_{i_1i_2i_3i_4}) \circ c(e_j) = 3 c(e_{i_1i_2i_3i_4})
\]
and hence $\sum_j c(e_j) \circ (\theta) \circ c(e_j) = 3 c(\theta)$. We conclude that 
\[
\sum_j c(e_j) \circ c(\beta_3)^2 \circ c(e_j) = - 5 |\beta_3|^2 + 3 c(\theta) = 3 c(\beta_3)^2 - 8 |\beta_3|^2
\]
as claimed.

The proof of~\eqref{conjugate-beta3} is similar to what we just did with $\theta$. In $\sum c(e_j) \circ c(\beta_3) \circ c(e_j)$ we  encounter terms of the form 
\[
c(e_j) \circ c(e_{i_1i_2i_3}) \circ c(e_j)
\]
for $i_1 <i_2< i_3$. If $j$ is distinct from $i_1,i_2,i_3$ the Clifford relations give us
\[
c(e_j) \circ c(e_{i_1i_2i_3}) \circ c(e_j) = c(e_{i_1i_2i_3})
\]
whilst if $j$ equals one of the other indices, we get
\[
c(e_j) \circ c(e_{i_1i_2i_3}) \circ c(e_j) = -c(e_{i_1i_2i_3})
\]
For given $i_1<i_2<i_3$, there are two values $j$ for which the first alternative occurs, and three for which the second occurs. The upshot is that 
\[
\sum_j c(e_j) \circ c(\beta_3) \circ c(e_j) = - c(\beta_3)\qedhere
\]
\end{proof}

Lemma~\ref{zeroth-order-5D} shows that for the 5-dimensional Seiberg--Witten equations, the zeroth order term can have potentially either sign. If $\beta_3=0$ then $Q(\beta)$ is non-negative, but if $\beta_5=0$ then $Q(\beta)$ is non-positive. In  six dimensions things are potentially even more troublesome, at least from the point of view of the maximum principle, since $Q(\beta)$ is necessarily non-positive. This follows from the next lemma. 

\begin{lemma}
On a 6-manifold, with $\beta = \beta_3$ purely of degree 3, we have $Q(\beta) = - 2|\beta|^2$.
\end{lemma} 

\begin{proof}
The proof is very similar to that of Lemma~\ref{zeroth-order-5D}. 
First notice that $c(\beta) = c(\beta)^*$. Next we need the two identities:
\begin{align*}
\sum c(e_j) \circ c(\beta)^2 \circ c(e_j) &= 2 c(\beta)^2 - 8|\beta|^2,\\
\sum c(e_j) \circ c(\beta) \circ c(e_j) & = 0.
\end{align*}
From here the calculation is essentially identical to that above and so we suppress the details. 
\end{proof}

For the record we also give the formula for $Q(\beta)$ in 8-dimensions. Again, $Q(\beta)$ is non-positive. The proof is essentially the same as for 5-dimensions and so we skip the details.

\begin{lemma}\label{Q-8D}
On an 8-manifold, with $\beta = \beta_3$ purely of degree 3, we have that
\[
Q(\beta) = - 2 c(\beta)^2 - 4|\beta|^2
\]
\end{lemma}

\section{A~priori estimates}\label{analysis}

In this section we prove some a priori estimates for solutions to the 8-dimensional Seiberg--Witten equations. (The same techniques apply to all dimensions, we focus on~8 just to be concrete.) 

\begin{remark}\label{harmonic-part}
Before we begin with the estimates, we make an important remark about how compactness can, and indeed does, fail. Suppose we have solutions $(\phi, A, \beta)$ and suppose, moreover, that there is a harmonic 3-form $\theta$ with the property that $c(\theta)(\phi)=0$. Then for any value of $t \in \R$, $(\phi, A, \beta + t\theta)$ is again a solution. This phenomenon genuinely arises in the examples found by the second author, which will appear in~\cite{Ghosh}. 

One way to escape this non-compactness phenomenon is prescribe the harmonic part of $\beta$. This is a finite dimensional constraint and so does not affect the ellipticity of the equations, only their index. 

An alternative to this is hinted at in Remark~\ref{gerbes}. Note first that something similar happens in the standard Seiberg--Witten story, when proving estimates for the connection $A$. There one first uses gauge transformations of the form $f= e^iu$ to put the connection $A$ in Coulomb gauge with respect to some reference, i.e. $A = A_0 + a$ where $\diff^*a=0$. The elliptic estimates for $\diff + \diff^*$ are enough to control $a$ modulo its harmonic part. To control the harmonic part of $a$, one needs to use so-called ``large'' gauge transformations $f \colon M \to S^1$ which are not homotopic to the identity. In order to replicate this kind of reasoning with $\beta$ however, we would need $\beta$ to be a connection form in a gerbe, on which we could act with gauge transformations. (In this version we would not need $\diff^*\beta$ to appear in the curvature equation; we would instead arrange for it to vanish by gauge transformations.) However for this to work, we would need to know how to define ``spinors with values in a gerbe'' and define the Dirac operator coupled to $\beta$.   
\end{remark}

\begin{remark}
The estimates we prove suggest that to produce compactness results, the crucial quantity to control will be the $L^8$ norm of either $\phi$ or $\beta$. This first appears in Theorem~\ref{Lp-compactness-result}, which shows that given a sequence of solutions with a uniform $L^{8+\epsilon}$ bound on either $\beta$ or on both $\phi$ and $\beta_h$ (the harmonic part of $\beta$) then there is a subsequence which converges modulo gauge. The importance of $L^8$ is also clear from Corollary~\ref{epsilon-regularity}, which is an $\epsilon$-regularity result for $\phi$ and $\beta$, where $\epsilon$ refers to the sum of their $L^8$-norms.
\end{remark}

We now move on to the details. Let $(M,g)$ be a compact 8-dimensional Riemannian manifold with a $\Spinc$ structure. Recall that the 8-dimensional Seiberg--Witten equations are for $\phi \in \Gamma(S_+)$, $A \in \A$ and $\beta \in \Omega^3$. They read
\begin{align}
\big( 
D_A + (1+i) c(\beta)
\big)(\phi)
	&=
		0,
		\label{8D-Dirac}\\
F_A + 2i \diff^*\beta + 2\diff^+ \beta
	&=
		q(\phi).
		\label{8D-curvature}
\end{align}
We will prove a priori estimates for solutions to~\eqref{8D-Dirac} and \eqref{8D-curvature}. We use throughout the convention that $C$ denotes a constant which can change from line to line, but remains at all times independent of $(\phi,A,\beta)$. We begin with the analogue of the standard $C^0$ estimate on $\phi$, which is the foundation of the compactness results for Seiberg--Witten equations in dimensions 3 and 4. In our situation, the same argument gives a bound on $\phi$ in terms of~$\beta$:

\begin{lemma}\label{C0-estimate-phi}
There is a constant $C$ such that if $(\phi, A, \beta)$ solve~\eqref{8D-Dirac} and~\eqref{8D-curvature} then 
\[
\| \phi\|_{C^0} \leq C \left( 1 + \| \beta \|_{C^0} \right)
\]
\end{lemma}
\begin{proof}
We begin by noting that for any unitary connection $\nabla$ in $S_+$, 
\begin{equation}
\frac{1}{2}\Delta |\phi|^2
=
\left\langle \nabla^*\nabla \phi, \phi \right\rangle - |\nabla \phi|^2
\label{Laplacian-nabla}
\end{equation}
(Here $\Delta = \diff^*\diff$ is the non-negative Laplacian on functions.) We will apply this with $\nabla = \nabla_{A,\beta}$.

Thanks to~\eqref{8D-curvature}, the Weitzenb\"ock formula for $D_{A,\beta} = D_A + (1+i) c(\beta)$ reads
\begin{equation}
D_{A,\beta}^*D_{A,\beta}
	=
		\nabla_{A,\beta}^* \nabla_{A,\beta}
		+ 
		\frac{s}{4}
		+ 
		\frac{1}{2} E_\phi
		+ 
		Q(\beta).
\label{another-Weitzenbock-formula}
\end{equation}
(Recall our notation~\eqref{Ephi} that $E_\phi$ denotes the trace-free part of projection along $\phi$.) Note that $Q(\beta)$ is made algebraically and quadratically from $\beta$; it is the last two terms in the statement of~Proposition~\ref{Weitzenbock-proposition}. This means that there is a constant $C$ such that for all $\psi$ and $\beta$, we have $|Q(\beta)(\psi)| \leq C |\beta|^2|\psi|$. 

We now apply~\eqref{another-Weitzenbock-formula} to $\phi$ and take the inner-product with $\phi$. Together with the Dirac eqaution~\eqref{8D-Dirac} and the fact that $\left\langle E_\phi(\phi),\phi \right\rangle= \frac{7}{8} |\phi|^2$, this gives
\begin{align}
\left\langle\nabla_{A,\beta}^*\nabla_{A,\beta} \phi, \phi \right\rangle
	&=
		-\frac{s}{4}|\phi|^2
		- \frac{1}{2} \left\langle E_\phi(\phi) , \phi \right\rangle
		- \left\langle Q(\beta)(\phi),\phi  \right\rangle,
		\nonumber\\
	& \leq
		C(1+ |\beta|^2)|\phi|^2 - \frac{7}{16}|\phi|^4.
		\label{rough-Laplacian-bound}
\end{align}
Combining with~\eqref{Laplacian-nabla} this gives
\begin{equation}
\Delta |\phi|^2 \leq C(1+|\beta|^2)|\phi|^2 - \frac{7}{8} |\phi|^4.
\label{Laplacian-phi-estimate}
\end{equation}
At a maximum of $|\phi|^2$, $0 \leq \Delta |\phi|^2$ (recall we are using the non-negative Laplacian) and so at that maximum point $p \in M$,
\[
\|\phi\|_{C^0}^2 
	= |\phi(p)|^2 \leq C(1 + |\beta(p)|^2) 
		\leq C\left(1+\| \beta\|^2_{C^0}\right)
\]
This is equivalent to the bound in the statement. 
\end{proof}

Having shown how to bound $\phi$ by $\beta$, we now consider the converse and bound $\beta$ by~$\phi$ (modulo the harmonic part of $\beta$, as per Remark~\ref{harmonic-part}). In the next lemma we use the following notation. Given a form $\beta \in \Omega^*$, we write $\beta_h$ for the harmonic part of $\beta$. Meanwhile we write $G \colon \Omega^* \to \Omega^*$ for the Green's operator of the Hodge Laplacian $\Delta = \diff^*\diff + \diff \diff^*$. So $\Delta G(\beta) = \beta - \beta_h = G(\Delta \beta)$. 

\begin{lemma}\label{beta-Green}
If $(\phi, A, \beta)$ solve~\eqref{8D-curvature} then $\beta - \beta_h$ is determined by $\phi$ via
\[
\beta - \beta_h= -\frac{i}{2} \diff\left(  G ( q(\phi)_2)\right) - * \diff\left(  G ( q(\phi)_4)\right)
\]
Here, $q(\phi) = q(\phi)_2 + q(\phi)_4$ where $q(\phi)_2\in i \Omega^2$ and $q(\phi)_4 \in \Omega^4_+$,.
\end{lemma}

\begin{proof}
Applying $\diff$ to~\eqref{8D-curvature}, splitting according to the degree of terms, and using the fact that $\diff F_A = 0$, we obtain two equations:
\begin{align*}
\diff \diff^* \beta & = - \frac{i}{2} \diff q(\phi)_2,\\
\diff \left( * \diff \beta \right) & =  \diff q(\phi)_4.
\end{align*}
Now using that $\diff^* = - * \diff~*$, these equations imply that 
\begin{equation}
\Delta \beta = - \frac{i}{2} \diff q(\phi)_2 - *\left( \diff q(\phi)_4\right)
\label{Delta-beta}
\end{equation}
The result now follows from applying $G$ and using the fact that it commutes with both $*$ and $\diff$.
\end{proof}

\begin{lemma}\label{Lp1-estimate-beta}
For any $p \in (0,\infty)$ there is a constant $C$ such that any solution $(\phi,A, \beta)$ of~\eqref{8D-curvature} we have
\[
\| \beta\|_{L^p_1} 
	\leq 
		C\left( 
		\| \phi \|_{L^{2p}}^2 + \| \beta_h\|
		\right)
\]
(where, at the expense of changing $C$ we can use any norm on harmonic 3-forms, since they are all equivalent). 
\end{lemma}

\begin{proof}
We have the following inequalities (where the constant $C$ may change from step to step, but remains always independent of $\beta$ and $\phi$):
\begin{gather*}
\| \diff \left( G(q(\phi)_2\right) \|_{L^p_1}
	\leq		C \| G (q(\phi)_2)\|_{L^p_2}	
	\leq 
		C \| q(\phi)\|_{L^p}
	\leq 
		C\| \phi \|_{L^2p}^2\\
\| * \diff \left( G(q(\phi)_4)\right) \|_{L^p_1}
	\leq
		C \| G(q(\phi)_4)\|_{L^p_2}
	\leq
		C \| q(\phi)\|_{L^p_2}
	\leq 
		C\| \phi \|_{L^2p}^2
\end{gather*}
The result now follows from Lemma~\ref{beta-Green}.
\end{proof}

We next prove that an $L^p$ bound on $\beta$, for $p>8$ gives a $C^0$-bound on $\phi$. 

\begin{proposition}\label{Lp-beta-to-C0-phi}
Let $p>8$. There is a constant $C$ such that for any solution $(\phi,A,\beta)$ to equations~\eqref{8D-Dirac} and~\eqref{8D-curvature},
\[
\| \phi\|_{C^0} \leq C \| \beta\|_{L^p}
\]
\end{proposition}

\begin{proof}
Recall~\eqref{Laplacian-phi-estimate}) which says that
\[
\Delta |\phi|^2 \leq C\left(|\beta|^2 +1 \right) |\phi|^2 - \frac{7}{8}|\phi|^4
\]
To ease notation, put $f = |\phi|^2$ and $g= C(|\beta|^2 + 1)$ so that
\[
\Delta f \leq gf - \frac{7}{8} f^2
\]
with $f \geq 0$. We now apply Moser iteration to this inequality. This is essentially standard, so we just give an outline. 

Multiply through by $f^{q+1}$ and integrate by parts to obtain
\[
\int f^{q+1} \Delta f
	=
	\frac{4(q+1)}{(q+2)^2} \int \left| \nabla \left( f^{\frac{q+2}{2}}\right)\right|^2
	\leq
	\int |g| f^{q+2} - \frac{7}{8} \int f^{q+3}
\]
It follows that
\begin{equation}
\frac{4(q+1)}{(q+2)^2} 
\left\| f^{\frac{q+2}{2}}\right\|_{L^2_1}^2
	\leq 
	\int \left( |g| + 4 \right)f^{q+2} - \frac{7}{8} \int f^{q+3}
\label{pre-moser}
\end{equation}

In the first instance, we simply drop the negative term on the right-hand side of~\eqref{pre-moser}. Using the 8-dimensional Sobolev embedding $L^2_1 \hookrightarrow L^{8/3}$ we have
\begin{equation}
\frac{4(q+1)}{(q+2)^2} \left\| f\right\|^{q+2}_{L^{4(q+2)}/3}
	\leq
		C \int \left( |g| + 4 \right)f^{q+2}
\label{iterate-me}
\end{equation}
Now we apply H\"older's inequality to the right-hand side with $m,n>1$ chosen so that $1/m+1/n = 1$ and we also choose $m>4$. This gives
\[
\frac{4(q+1)}{(q+2)^2} \left\| f\right\|^{p+2}_{L^{4(q+2)}/3}
	\leq 
	C \left\| |g|+ 4 \right\|_{L^m} \| f \|_{L^{n(q+2)}}^{q+2}
\]
Since $m>4$, $n<4/3$ and so we have bounded a higher $L^k$-norm of $f$ by a lower one. We can now iterate this, starting from $q=0$, to obtain a $C^0$ bound of $f$ in terms of a $L^{2n}$ bound of $f$ and an $L^m$ bound on $|g|+4$. One must pay attention to the dependence on $q$ of the constants in~\eqref{iterate-me}, but this is standard for Moser iteration arguments. The upshot is that there is a constant $C$ such that
\[
\| f\|_{C^0} \leq C \| |g|+4\|_{L^m} \| f\|_{L^{2n}}
\] 

We now bound $f$ directly in $L^{2n}$. To do this we integrate $f\Delta f \leq gf^2 - \frac{7}{8}f^3$ to obtain
\[
\frac{7}{8}\int f^3 \leq \int gf^2 \leq  \frac{1}{3}\int |g|^3 + \frac{2}{3} \int f^3
\]
(where we have used H\"older's inequality). So $\int f^{3} \leq \frac{8}{5} \int |g|^3$. Since $2n< 3$ and $m>4$, we see that $\| f\|_{L^{2n}} \leq C \| g \|_{L^m}$. Putting the pieces together, we see that 
\[
\| f\|_{C^0} \leq C \| |g| + 4 \|_{L^m}
\]
for any choice of $m>4$. Since $g = |\beta|^2$ this gives the estimate in the statement of the Lemma.
\end{proof}

We can combine all of these estimates to show that an $L^p$ bound on both $\phi$ and $\beta_h$ or simply on $\beta$ alone is enough to deduce a $C^0$ bound on $\beta$ (and hence $\phi$ by Lemma~\ref{C0-estimate-phi}).

\begin{proposition}\label{L8-plus-epsilon-wins}
Let $p>8$. There exists a constant $C$ such that for any solution $(\phi,A, \beta)$ to equations~\eqref{8D-Dirac} and~\eqref{8D-curvature},
\begin{align}
 \| \beta \|_{C^0} & \leq C \| \beta \|^2_{L^p} \label{Lp-beta-bounds}\\
 \| \beta \|_{C^0} &\leq C \left( \| \phi \|^4_{L^p} + \| \beta_h \|^2\right)\label{Lp-phi-bounds}
\end{align}
\end{proposition}
\begin{proof}
We start with the proof of~\eqref{Lp-beta-bounds}. Proposition~\ref{Lp-beta-to-C0-phi} already gives $\| \phi\|_{C^0} \leq C \| \beta\|_{L^p}$ which in turn implies the weaker estimate
\[
\| \phi\|_{L^{2p}} \leq C \| \beta\|_{L^p}
\]
Now Lemma~\ref{Lp1-estimate-beta} together with Sobolev embedding $L^p_1 \hookrightarrow C^0$ implies~\eqref{Lp-beta-bounds}. 

For~\eqref{Lp-phi-bounds} we combine Lemma~\ref{Lp1-estimate-beta} (with $p$ replaced by $p/2$) with the Sobolev embedding $L^{p/2}_1 \hookrightarrow L^q$ where $\frac{8}{q} = \frac{16}{p} -1$. For this $q$ we have
\[
\| \beta \|_{L^q} \leq C \left( \| \phi \|^2_{L^p} + \| \beta_h \|\right)
\]
Since $p>8$, we also have that $q >8$. Now Proposition~\ref{Lp-beta-to-C0-phi} gives
\[
\| \phi \|_{C^0} \leq C  \left( \| \phi \|^2_{L^p} + \| \beta_h \|\right)
\]
This implies the weaker bound
\[
\| \phi \|_{L^{2p}} \leq C  \left( \| \phi \|^2_{L^p} + \| \beta_h \|\right)
\]
which we can now put back in Lemma~\ref{Lp1-estimate-beta} to conclude that 
\[
\| \beta\|_{L^p_1} \leq C \left( \| \phi \|^4_{L^p} + \| \beta_h \|^2\right)
\]
From here~\eqref{Lp-phi-bounds} follows from the Sobolev embedding $L^p_1 \hookrightarrow C^0$.
\end{proof}

We can now give our main compactness result.

\begin{theorem}\label{Lp-compactness-result}
Let $(\phi_n,A_,\beta_n)$ be a sequence of solutions to equations~\eqref{8D-Dirac} and~\eqref{8D-curvature}. Suppose one of the two following conditions is satisfied. 
\begin{enumerate}
\item
There is $p>8$ and a constant $C$ such that for all $n$, 
\[
\| \beta_n \|_{L^p} \leq C
\]
\item
There is $p>8$ and a constant $C$ such that for all $n$, 
\[
\| \phi_n\|_{L^p} \leq C
\]
\emph{and} $\| (\beta_n)_h \| \leq C$ (where $\beta_h$ denotes the harmonic part of $\beta$ and we use any fixed choice of norm on the finite dimensional space of harmonic 3-forms).
\end{enumerate}
Then a subsequence of $(\phi_n,A_n,\beta_n)$ converges modulo gauge in $C^{\infty}$.
\end{theorem}

\begin{proof}
We closely follow the proof of compactness for the 4-dimensional Seiberg--Witten equations. Since this is standard, we don't give all the details. Moreover, we only give the proof assuming $\| \beta_n \|_{L^p} \leq C$ since the other part is almost identical. 

First, observe that the harmonic part of $\beta_n$ is uniformly bounded. This is because $\beta_n$ is bounded in $L^2$ and $(\beta_n)_h$ is the $L^2$-orthongonal projection of $\beta_n$ to $\ker \Delta$. So $\| (\beta_n)_h\|_{L^2}\leq C$ and hence $\| (\beta_n)_h \| \leq C$ for any choice of norm (since $\ker \Delta$ is finite dimensional). 

We will now show that $F_{A_n}$ is bounded in $C^0$. By Proposition~\ref{L8-plus-epsilon-wins}, we have $\| \beta_n\|_{C^0} \leq C$ for all $n$. Now by Lemma~\ref{C0-estimate-phi} we have $\| \phi_n \|_{C^0} \leq C$ for all $n$. Taking the $i\Omega^2$-component of~\eqref{8D-curvature} we have $F_{A_n} = q(\phi)_2$ and so $\| F_{A_n} \|_{C^0} \leq C$. 

Next write $A_n = A_0 + a_n$ for some reference connection $A_0$ and imaginary 1-form $a_n$. We explain how to ensure that $a_n$ is bounded in $L^p_1$. We act by gauge transformations so that $d^*a_n = 0$ and, moreover, so that the harmonic part of $a_n$ is uniformly bounded. (This is identical to the familiar 4-dimensional compactness result.) Now the fact that $\diff a_n = F_{A_n} - F_{A_0}$ means $\| \diff a_n + \diff^*a_n \|_{L^p} \leq C$. The $L^p_1$ elliptic estimate for $\diff +\diff^*$, together with the uniform bound on the harmonic part of $a_n$ now implies $\|a_n\|_{L^p_1} \leq C$ for any $p$ (remembering that $C$ depends on the choice of $p$). 

We next show that $\phi_n$ is uniformly bounded in $L^p_1$. We use the Dirac equation which can be written
\begin{equation}
D_{A_0}(\phi_n) = - c(a_n) \phi_n - c(\beta_n) \phi_n
\label{Dirac-for-bootstrap}
\end{equation}
The right-hand side is uniformly bounded in $L^p$. Moreover, $\phi_n$ is uniformly bounded in $L^2$ and so its $D_{A_0}$ harmonic part is also uniformly bounded. (The $D_{A_0}$ harmonic part of $\phi_n$ is the $L^2$-orthogonal projection to $\ker D_{A_0}$.) The $L^p_1$ elliptic estimate for $D_{A_0}$ now gives that $\|\phi_n\|_{L^p_1} \leq C$ for any choice of $p$ (where $C$ depends on $p$).

We can immediately bootstrap this to a uniform $L^p_2$ bound on $\phi_n$. By Lemma~\ref{Lp1-estimate-beta} the uniform $C^0$ bound on $\phi_n$ and the uniform bound on $(\beta_n)_h$ give a uniform $L^p_1$ bound on $\beta_n$.  Taking $p$ large enough so that Sobolev multiplication holds, the right-hand side of~\eqref{Dirac-for-bootstrap} is uniformly bounded in $L^p_1$ (since this is true for $a_n$, $\beta_n$ and $\phi_n$). Now the $L^p_2$ elliptic estimate for $D_{A_0}$, and the fact that the $D_{A_0}$-harmonic part of $\phi_n$ is bounded, imply that $\| \phi_n \|_{L^p_2} \leq C$.  

The next step is to show $\beta_n$ is uniformly bounded in $L^p_2$. Again with $p$ high enough that Sobolev multiplication holds, we have $\|q(\phi_n)\|_{L^p_1} \leq C$. By Lemma~\ref{beta-Green},
\[
\beta_n= (\beta_n)_h - \frac{i}{2} \diff (G(q(\phi_n)_2)) - * \diff (G (q(\phi_n)_4))
\]
So
\[
\| \beta_n \|_{L^p_2} 
	\leq 
		\| (\beta_n)_h\|_{L^p_2}
		+ 
		\frac{3}{2} \| q(\phi_n) \|_{L^p_1}
	\leq 
	C
\]
We have used here the fact that $G \colon L^p_1 \to L^p_3$ is bounded and also the fact that $(\beta_n)_h$ is uniformly bounded in any norm we choose.

To show $a_n$ is uniformly bounded in $L^p_2$ we argue just as for the $L^p_1$ bound on $a_n$, but with the new information that $F_{A_n} = q(\phi_n)_2$ is now known to be bounded in $L^p_1$. This, together with the $L^p_2$ elliptic estimate for $\diff + \diff^*$ shows $\|a_n\|_{L^p_2} \leq C$.

At this point we have $\phi_n, a_n, \beta_n$ all bounded in $L^p_2$ and we can iterate the whole discussion to obtain uniform bounds in $L^p_k$ for any choice of $k$. From here we can pass to a subsequence which converges in $L^p_k$ for every $k$ and so converges in $C^{\infty}$. 
\end{proof}

We finish this section by showing that together $\beta$ and $\phi$ satisfy a partial differential inequality of the kind which appears in $\epsilon$-regularity results.

\begin{proposition}\label{PD-inequality}
There is a constant $C>0$ such that for any solutions $(\phi, A, \beta)$ to the equations~\eqref{8D-Dirac} and~\eqref{8D-curvature} we have
\[
\Delta \left( |\beta|^2 + |\phi|^2 \right)
 	\leq
		C \left(|\beta|^2 + |\phi|^2 + |\phi|^2|\beta|^2\right)
		-
		\frac{7}{8}|\phi|^4
\]
(Again, recall $\Delta$ is the non-negative Laplacian here.)
\end{proposition}

\begin{proof}
We recall the Bochner formula for the Hodge Laplacian which gives
\[
\frac{1}{2} \Delta |\beta|^2
	=
		\left\langle \Delta \beta, \beta \right\rangle
		-
		|\nabla \beta|^2
		+
		R(\beta)
\]
where $\Delta=\diff^*\diff + \diff \diff^*$ is the Hodge Laplacian, $\nabla$ is the Levi-Civita connection and $R(\beta)$ is a certain contraction of the Riemannian curvature tensor and $\beta$ (whose precise form will not matter in what follows). This implies that
\[
\frac{1}{2} \Delta |\beta|^2
	\leq
		|\Delta \beta| |\beta| + C|\beta|^2
\]
for some constant $C$. 

Now from~\eqref{Delta-beta}, $|\Delta \beta| \leq |\nabla q(\phi)|$.The connection $\nabla_A$ is a spin connection in the sense that it preserves Clifford multiplication: for any form $\gamma$, $\nabla_A c(\gamma) = c (\nabla \gamma)$. Since $q(\phi) = c^{-1}(E_\phi)$ and $E_\phi$ is quadratic in $\phi$ (see~\eqref{Ephi}), there is a constant $C$ such that for any $A\in A$ and $\phi$
\[
|\nabla q(\phi)| \leq C |\nabla_A \phi||\phi|
\]
Next recall that $\nabla_{A,\beta} = \nabla_A + B$ where, in a local frame, 
\[
B = \sum e_j \otimes \left( 
					c(e_j) \otimes c(\beta) + c(\beta)^* \otimes c(e_j)
				\right)
\]
This means that $|\nabla_{A}\phi| \leq |\nabla_{A,\beta} \phi| + C|\beta||\phi|$ and hence
\begin{align}
\Delta |\beta|^2 
	&\leq
		C \left( 
			|\nabla_{A,\beta}\phi| |\phi| |\beta| + |\phi|^2 |\beta|^2 + |\beta|^2
		\right)
		\nonumber\\
	&\leq
			\frac{1}{2} |\nabla_{A,\beta}\phi|^2 
			+
			C\left( 
			1 + |\phi|^2
			\right)|\beta|^2
	\label{Delta-norm-beta}
\end{align}
Here we have used the inequality $ab \leq \frac{1}{2C} a^2 + \frac{C}{2} b^2$ with $a = |\nabla_{A,\beta}\phi|$ and $b = |\phi||\beta|$. 

Meanwhile, in the course of the proof of Lemma~\ref{C0-estimate-phi} (when combining~\eqref{Laplacian-nabla} and~\eqref{rough-Laplacian-bound}) we showed that
\begin{align}
\Delta |\phi|^2
	&\leq
		C(1 + |\beta|^2)|\phi|^2 
		- 
		\frac{7}{8} |\phi|^4 
		- 
		\frac{1}{2}|\nabla_{A,\beta} \phi|^2 
\label{Delta-norm-phi}
\end{align}
(This is slightly stronger than what is stated in the line after~\eqref{rough-Laplacian-bound} where we dropped the term $-|\nabla_{A,\beta} \phi|^2$ from the right-hand side.)
Now adding~\eqref{Delta-norm-beta} and~\eqref{Delta-norm-phi} gives the result.
\end{proof}

Notice that the inequality in Proposition~\ref{PD-inequality} implies immediately that, for some constants $A,B$ we have
\[
\Delta \left(|\phi|^2 + |\beta|^2\right) \leq A \left( |\phi|^2 + |\beta|^2\right)^2 + B\left(|\phi|^2 +|\beta|^2 \right)
\]
This is the shape of inequality, namely $\Delta f \leq Af^2 + Bf$, which normally arises in the context of $\epsilon$-regularity. The standard argument used to prove $\epsilon$-regularity yields our final estimate. 
\begin{corollary}\label{epsilon-regularity}
Let $B_r \subset M$ be a geodesic ball of radius $r$ centred at a point $x$. There exist constants $\epsilon >0$ and $C>0$ such that if 
\[
\int_{B_r} \left( |\phi|^8 + |\beta|^8 \right)\leq \epsilon
\]
then
\[
\|\phi\|_{C^0(B_{r/2})} + \|\beta\|_{C^0(B_{r/2})} \leq Cr^{-2}  \int_{B_r} \left( |\phi|^8 + |\beta|^8 \right)^{1/8}
\]
\end{corollary}

\section{An energy identity}\label{energy-identity}

In 4-dimensions, on a compact manifold with $\Spinc$-structure, solutions $(A,\phi)$ of the Seiberg--Witten equations are absolute minima of the following energy functional:
\[
\E(A,\phi)
	=
		\int |\nabla_A \phi|^2 + \frac{1}{2}|F_A|^2 + \frac{1}{8}\left( |\phi|^2 + s\right)^2 
		\dvol	
\]
In this section we prove an analogous result for solutions $(A,\beta,\phi)$ to the 8-dimensional Seiberg--Witten equations. (Similar results hold in all dimensions, with different constants appearing in various places; we work in dimension~8 to keep the discussion concrete.)

\begin{definition}
Let $M$ be an oriented Riemannian 8-manifold with a $\Spinc$-structure. Given a connection $A \in \A$, a 3-form $\beta \in \Omega^3$ and a spinor $\phi \in \Gamma(S_+)$ we define the energy of $(A,\beta,\phi)$ to be 
\[
\E(A,\beta,\phi)
	=
		\int |\nabla_{A,\beta}\phi|^2
		+
		2|F_A|^2 + 8| \diff^*\beta|^2
		+
		8|d \beta|^2 
		+
		\frac{7}{32}\left(
		|\phi|^2 + \frac{4 s}{7} \right)^2
		-
		2|c(\beta)\phi|^2
		- 
		4 |\beta|^2 |\phi|^2
\]
(Here $s$ is the scalar curvature of $M$ and $\nabla_{A,\beta}$ is the connection in $S_+$ determined by $A$ and $\beta$ which appears in Proposition~\ref{4m-dim-Weitzenbock-proposition}.)
\end{definition}

\begin{theorem}\label{energy-minima}
Let $M$ be an oriented Riemannian 8-manifold with a $\Spinc$-structure. For any $(A,\beta,\phi)$, we have the inequality
\[
\E(A,\beta,\phi) \geq \frac{1}{14}\int s^2
\]
with equality if and only if $(A,\beta,\phi)$ solve the 8-dimensional Sieberg--Witten equations. 
\end{theorem}

The proof takes up the rest of this section. We begin with the following lemma.  

\begin{lemma}\label{innerpoduct-relations}
On an 8-dimensional manifold, let $i\omega \in i\Lambda^2$ and $\theta \in \Lambda^4_+$, Then for any $\phi \in S_+$, 
\begin{align*}
\left\langle c(i\omega)\phi, \phi \right\rangle
	&=
		8 \left\langle i\omega, q(\phi)_2 \right\rangle\\
\left\langle c(\theta),\phi, \phi \right\rangle
	&=
		16 \left\langle \theta, q(\phi)_4 \right\rangle
\end{align*}
Here we write $q(\phi) =  q(\phi)_2 + q(\phi)_4$ for the degree~2 and~4 parts of the differential form $q(\phi)$. The inner-products on the left are in $S_+$ whilst on the right they are in~$\Lambda^*$. 
\end{lemma}
\begin{proof}
Recall that Clifford multiplication gives an isomorphism
\[
c \colon i\Lambda^2 \oplus \Lambda^4_+ \to i\su(S_+)
\]
In dimension~8, the bundle $S_+$ carries a real structure. Accordingly,  $i\su(S_+) = V_1 \oplus V_2$ splits into those endomorphisms which commute with the real structure and those which anti-commute. These correspond under $c$ to the action of $\Lambda^4_+$ and  $i\Lambda^2$ respectively. The maps $c \colon i\Lambda^2 \to V_1$ and $c \colon \Lambda^4_+ \to V_2$ now correspond to isomorphisms of \emph{irreducible} $\Spin(8)$ representations. It follows from Schur's Lemma that there are constants $a_1,a_2$ such that for all $i\omega \in i\Lambda^2$ and $\theta \in \Lambda^4_+$,
\begin{align*}
|c(i\omega)|^2 &= a_1|\omega|^2\\
|c(\theta)|^2 & = a_2 |\theta|^2
\end{align*}

To find $a_1, a_2$ we need to just compute examples. One such example is provided by a Kähler manifold with Kähler form $\omega$. We have $S_+ = \Omega^{0,0}\oplus \Omega^{0,2} \oplus \Omega^{0,4}$. With respect to this splitting,
\[
c(i\omega)
	=
		\begin{pmatrix}
		4 & 0 & 0\\
		0 & 0 & 0\\
		0 & 0 & -4
		\end{pmatrix},\qquad
c(\omega^2)
	=
		\begin{pmatrix}
		-12 & 0 & 0\\
		0 & 4 & 0\\
		0 & 0 & -12
		\end{pmatrix}
\]
From this we see that
\begin{align*}
\tr \left( c(i\omega)^2\right) &= 8 |\omega|^2\\
\tr \left( c(\omega^2)^2\right) &= 16 |\omega^2|^2
\end{align*}
So $a_1 = 8$ and $a_2 = 16$. 

We can now compute
\begin{align*}
\left\langle \theta, q(\phi)_4 \right\rangle
	&=
		\frac{1}{16} \left\langle c(\theta), E_\phi \right\rangle\\
	&=	
		\frac{1}{16} \tr \left( c(\theta) \circ \phi^* \otimes \phi \right)\\
	&=
		\frac{1}{16}
			\sum \left\langle e_i ,\phi  \right\rangle \left\langle c(\theta)(\phi), e_i  \right\rangle\\
	&=
		\frac{1}{16} \left\langle c(\theta)(\phi), \phi \right\rangle
\end{align*}
as claimed. (In the first line, we use the fact that $a_2= 16$, so that $c$ scales the inner-product by $16$. In the second line we use the fact that $c(\theta)$ is trace-free and so we can replace $E_\phi$ by $\phi^*\otimes \phi$; recall~\eqref{Ephi} for $E_\phi$.)  The proof for $i\omega$ is similar.
\end{proof}

\begin{proof}[Proof of Theorem~\ref{energy-minima}]
Recall the Weitzenb\"ock formula of Proposition~\ref{4m-dim-Weitzenbock-proposition} which (together with the formula of Lemma~\ref{Q-8D} for $Q(\beta)$) implies
\[
\int |D_{A,\beta} \phi|^2
=
		\int |\nabla_{A,\beta}\phi|^2
		+
		\frac{s}{4}|\phi|^2
		+
		\frac{1}{2}\left \langle c(F_A+ 2i \diff^*\beta +2d^+ \beta)(\phi), \phi \right\rangle
		-
		2 |c(\beta)(\phi)|^2
		-
		4 |\beta|^2 |\phi|^2
\]
By Lemma~\ref{innerpoduct-relations}, 
\[
\frac{1}{2}\left \langle c(F_A+ 2i \diff^*\beta +2d^+ \beta)(\phi), \phi \right\rangle
=
4 \left\langle 
F_A + 2i \diff^*\beta, q(\phi)_2
\right\rangle
+
16 \left\langle 
 \diff^+\beta, q(\phi)_4
 \right\rangle
\]
So, by completing the square we obtain
\begin{multline*}
\int |D_{A,\beta} \phi|^2
+
2 |F_A + 2i\diff^*\beta - q(\phi)_2|^2
+
4 |2d^+\beta - q(\phi)_4|^2
\\
	=
		\int  |\nabla_{A,\beta}\phi|^2
		+
		\frac{s}{4}|\phi|^2 
		+
		2 |F_A+2i\diff^*\beta|^2 +2 |q(\phi)_2|^2 \\
		+
		16 |\diff^+\beta|^2 + 4|q(\phi)_4|^2 
		-
		2 |c(\beta)(\phi)|^2
		-
		4 |\beta|^2 |\phi|^2
\end{multline*}
Next we use the fact that $F_A$ and $i\diff^*\beta$ are $L^2$-orthogonal (since $\diff F_A=0$). We also use
\[
0 = \int \diff (\beta \wedge \diff \beta) = \int |\diff^+ \beta|^2 - \int |\diff^-\beta|^2
\]
which shows that $\int |\diff^+\beta|^2 = \frac{1}{2} \int |\diff \beta|^2$. This gives
\begin{multline*}
\int |D_{A,\beta} \phi|^2
+
2 |F_A + 2i\diff^*\beta - q(\phi)_2|^2
+
4 |2d^+\beta - q(\phi)_4|^2
\\
	=
		\int  |\nabla_{A,\beta}\phi|^2
		+
		\frac{s}{4}|\phi|^2 
		+
		2 |F_A|^2+ 8|\diff^*\beta|^2 +2 |q(\phi)_2|^2 \\
		+
		8 |\diff \beta|^2 + 4|q(\phi)_4|^2 
		-
		2 |c(\beta)(\phi)|^2
		-
		4 |\beta|^2 |\phi|^2
\end{multline*}
Now, again by Lemma~\ref{innerpoduct-relations}, and the definition~\ref{Ephi} of $E_\phi$, we have
\[
2|q(\phi)_2|^2 + 4|q(\phi)_4|^2 
	= 
		\frac{1}{4}\left\langle E_\phi (\phi), \phi \right\rangle
	=
		\frac{7}{32}|\phi|^4
\]
So, completing the square on the terms involving $s|\phi|^2$ and $|\phi|^4$ we obtain 
\[
\int |D_{A,\beta} \phi|^2
+
2 |F_A + 2i\diff^*\beta - q(\phi)_2|^2
+
4 |2d^+\beta - q(\phi)_4|^2
	=
		\E(A,\beta,\phi)
		-
		\frac{1}{14} s^2
\]
The left-hand side is non-negative, and equal to zero if and only if $(A,\beta,\phi)$ solve the 8-dimensional Seiberg--Witten equations.
\end{proof}	

\small{
\bibliography{SW-bibliography} 

\begin{thebibliography}{10}

\bibitem{Atiyah-Singer}
{\sc Atiyah, M.~F., and Singer, I.~M.}
\newblock The index of elliptic operators. {III}.
\newblock {\em Ann. of Math. (2) 87\/} (1968), 546--604.

\bibitem{Bismut}
{\sc Bismut, J.-M.}
\newblock A local index theorem for non-{K}\"{a}hler manifolds.
\newblock {\em Math. Ann. 284}, 4 (1989), 681--699.

\bibitem{Fine-Ghosh-Ragini}
{\sc Fine, J., Ghosh, P., and Singhal, R.}
\newblock The {S}eiberg-{W}itten equations on {$G_2$} and
  {$\Spin(7)$}-manifolds, in preperation.

\bibitem{Friedrich}
{\sc Friedrich, T.}
\newblock {\em {D}irac operators in {R}iemannian geometry}, vol.~25 of {\em
  Graduate Studies in Mathematics}.
\newblock American Mathematical Society, Providence, RI, 2000.
\newblock Translated from the 1997 German original by Andreas Nestke.

\bibitem{Ghosh}
{\sc Ghosh, P.}
\newblock Solutions to the {S}eiberg-{W}itten equations in all dimensions, in
  preperation.

\bibitem{Lichnerowicz}
{\sc Lichnerowicz, A.}
\newblock Les spineurs en relativit\'{e} g\'{e}n\'{e}rale.
\newblock {\em Confer. Sem. Mat. Univ. Bari}, 79 (1962), 15.

\bibitem{Schrodinger}
{\sc Schr{\"o}dinger, E.}
\newblock Diracsches {Elektron} im {Schwerefeld}. {I}.
\newblock {\em Sitzungsber. Preu{{\ss}}. Akad. Wiss., Phys.-Math. Kl. 1932\/}
  (1932), 105--128.

\bibitem{Seiberg-Witten}
{\sc Seiberg, N., and Witten, E.}
\newblock Electric-magnetic duality, monopole condensation, and confinement in
  {N}=2 supersymmetric yang-mills theory.
\newblock {\em Nuclear Physics B 426}, 1 (1994), 19--52.

\bibitem{Smale}
{\sc Smale, S.}
\newblock On the structure of manifolds.
\newblock {\em Amer. J. Math. 84\/} (1962), 387--399.

\bibitem{Tanaka}
{\sc Tanaka, Y.}
\newblock {S}eiberg-{W}itten type equations on compact symplectic 6-manifolds,
  2014.
\newblock Preprint, arXiv:1407.1934.

\bibitem{Taubes1}
{\sc Taubes, C.~H.}
\newblock The {S}eiberg-{W}itten invariants and symplectic forms.
\newblock {\em Math. Res. Lett. 1}, 6 (1994), 809--822.

\bibitem{Taubes3}
{\sc Taubes, C.~H.}
\newblock {$\rm Gr\Rightarrow SW$}: from pseudo-holomorphic curves to
  {S}eiberg-{W}itten solutions [{MR}1728301 (2000i:53123)].
\newblock In {\em Seiberg {W}itten and {G}romov invariants for symplectic
  4-manifolds}, vol.~2 of {\em First Int. Press Lect. Ser.} Int. Press,
  Somerville, MA, 2000, pp.~163--273.

\bibitem{Taubes4}
{\sc Taubes, C.~H.}
\newblock {$\rm Gr=SW$}: counting curves and connections [1761081].
\newblock In {\em Seiberg {W}itten and {G}romov invariants for symplectic
  4-manifolds}, vol.~2 of {\em First Int. Press Lect. Ser.} Int. Press,
  Somerville, MA, 2000, pp.~275--401.

\bibitem{Taubes2}
{\sc Taubes, C.~H.}
\newblock {$\rm SW\Rightarrow Gr$}: from the {S}eiberg-{W}itten equations to
  pseudo-holomorphic curves [{MR}1362874 (97a:57033)].
\newblock In {\em Seiberg {W}itten and {G}romov invariants for symplectic
  4-manifolds}, vol.~2 of {\em First Int. Press Lect. Ser.} Int. Press,
  Somerville, MA, 2000, pp.~1--97.

\end{thebibliography}
\bibliographystyle{acm}
}

\end{document}